\documentclass[11pt,twoside]{article}
\usepackage[a4paper,hmargin=1in,vmargin=1in]{geometry}
\usepackage{graphicx,amsmath,amsfonts,latexsym,amssymb,amsthm}
\usepackage[nottoc,notlot,notlof]{tocbibind}
\usepackage[title,titletoc]{appendix} 
\usepackage{xcolor}
\usepackage{booktabs}
\usepackage{multirow}
\usepackage{bigdelim}
\usepackage{framed}
\usepackage[colorlinks,citecolor=blue,linkcolor=blue]{hyperref}
\usepackage{newfloat}
\usepackage{fancybox}
\usepackage{bbm}
\usepackage{ebgaramond}
\usepackage[T1]{fontenc}

\newtheoremstyle{plainsc}
{\topsep}
{\topsep}
{\itshape}
{}
{\large\scshape}
{}
{5pt plus 1pt minus 1pt}
{\thmname{#1}\thmnumber{ #2}\thmnote{ (#3)}.}

\newtheoremstyle{definitionsc}
{}
{}
{\normalfont}
{}
{\large\scshape}
{ }
{ }
{\thmname{#1}\thmnumber{ #2}\thmnote{ (#3)}.}

\newtheoremstyle{remarksc}
{0.5\topsep}
{0.5\topsep}
{\normalfont}
{}
{\large\itshape}
{}
{ }
{\thmname{#1}\thmnumber{ #2}\thmnote{ (#3)}.}

\newcommand{\R}{{\mathbb R}}
\newcommand{\Z}{{\mathbb Z}}
\newcommand{\RZ}{\R/\Z}
\newcommand{\I}{\mathrm{i}}

\newcommand{\dr}{\mathrm{d}}

\renewcommand{\Re}{\operatorname{Re}}
\renewcommand{\Im}{\operatorname{Im}}

\newcommand{\bbS}{\mathbb{S}}
\DeclareMathOperator{\arccosh}{arccosh}

\newcommand{\av}{\mathord{\bf av}}
\renewcommand{\tilde}{\widetilde}

\newcommand{\id}{\mathbbm{1}}
\theoremstyle{plainsc}
\newtheorem{theorem}{Theorem}[section]
 
 \newtheorem{cor}[theorem]{Corollary}

 \theoremstyle{definitionsc}

\theoremstyle{remarksc}
 \newtheorem{remark}[theorem]{Remark}

\definecolor{blue}{rgb}{0.05,0.2,0.7}

\renewcommand\footnotemark{}

\usepackage{fancyhdr}
\makeatletter
\renewcommand*\env@matrix[1][\arraystretch]{%
  \edef\arraystretch{#1}%
  \hskip -\arraycolsep
  \let\@ifnextchar\new@ifnextchar
  \array{*\c@MaxMatrixCols c}}
\makeatother
\DeclareFloatingEnvironment[fileext=lov,listname={Links to Videos}, name={Link to Video}, placement=htb, within=none]{video}

\begin{document}
\title{Computations of eigenvalues and resonances on perturbed hyperbolic surfaces with cusps
\footnote{All the videos accompanying this paper are available either at \href{http://michaellevitin.net/hyperbolic.html}{\texttt{michaellevitin.net/hyperbolic.html}} or as a dedicated \href{https://www.youtube.com/playlist?list=PLZB0Pfj9QSAIsKeCU8cWDcM-KZw0I6BsR}{\texttt{YouTube playlist}}}
\footnote{{\bf MSC classes:} 58J50, 35P25, 11F72, 65N25, 65N30}
}%
\author{%
Michael Levitin
\thanks{%
{\bf ML}: Department of Mathematics and Statistics, University of Reading, Reading RG6 6AX, UK; 
\href{mailto:m.levitin@reading.ac.uk}{\texttt{m.levitin@reading.ac.uk}}; 
\href{http://michaellevitin.net}{\texttt{michaellevitin.net}}%
}
\and 
Alexander Strohmaier
\thanks{%
{\bf AS}: School of Mathematics, University of Leeds, Leeds, LS2 9JT, UK; 
\href{mailto:a.strohmaier@leeds.ac.uk}{\texttt{a.strohmaier@leeds.ac.uk}}; 
\href{https://physicalsciences.leeds.ac.uk/staff/80/professor-alexander-strohmaier}{\texttt{physicalsciences.leeds.ac.uk/staff/80/professor-alexander-strohmaier}}%
}
}

\date{\small 30 December 2018; revised 17 June 2019; to appear in \emph{Int. Math. Res. Notices}}

\maketitle

\begin{abstract}
In this paper we describe a simple method that allows for a fast direct computation of the scattering matrix for a surface with hyperbolic cusps from the Neumann-to-Dirichlet map
 on the compact manifold with boundary obtained by removing the cusps. We illustrate that even if the Neumann-to-Dirichlet  map is obtained by a Finite Element Method (FEM) one can achieve good accuracy for the scattering matrix. We give various interesting examples of how this can be used to investigate the behaviour of resonances
under conformal perturbations or when moving in Teichm\"uller space. For example, based on numerical experiments we rediscover the four arithmetic surfaces of 
genus one with one cusp. This demonstrates that it is possible to identify arithmetic objects using FEM.\end{abstract}

\tableofcontents

\listoffigures
 
\listoftables

\listofvideos

\section{Introduction and setup}

Suppose that $\mathbb{H}=\{z=x+\I y \mid y>0\}$ is the upper half-plane with the  hyperbolic (of constant curvature $-1$)
metric 
\[
y^{-2} (\dr x^2 + \dr y^2).
\] 
The Riemannian measure is then  $y^{-2} \dr x \dr y$ and the $L^2$-inner product
is given by
\[
 \langle f, g \rangle = \int f(z) \overline{g(z)} y^{-2} \dr x \dr y.
\]
The metric Laplace operator
\[
\Delta=\Delta_\mathbb{H} = -y^2 \left( \partial_x^2 + \partial_y^2 \right)
\]
is essentially self-adjoint with domain $C^\infty_0(\mathbb{H})$, and later on we do not distinguish notationally  operators and their closures, if there is no danger of confusion.	

The map $(x,y) \mapsto (x+1,y)$ is an isometry of the upper half-space ,
and the quotient of the set $\mathbb{H}_a=\{z=x+\I y \mid y>a\}$ by this isometry results in a so-called hyperbolic cusp with height $a>0$.
Thus, such a cusp $Z^a$ is topologically equivalent to $\bbS^1 \times [a,\infty)$
and it is equipped with a metric of constant negative curvature.

\begin{figure}[htb!]
 \begin{center}\includegraphics{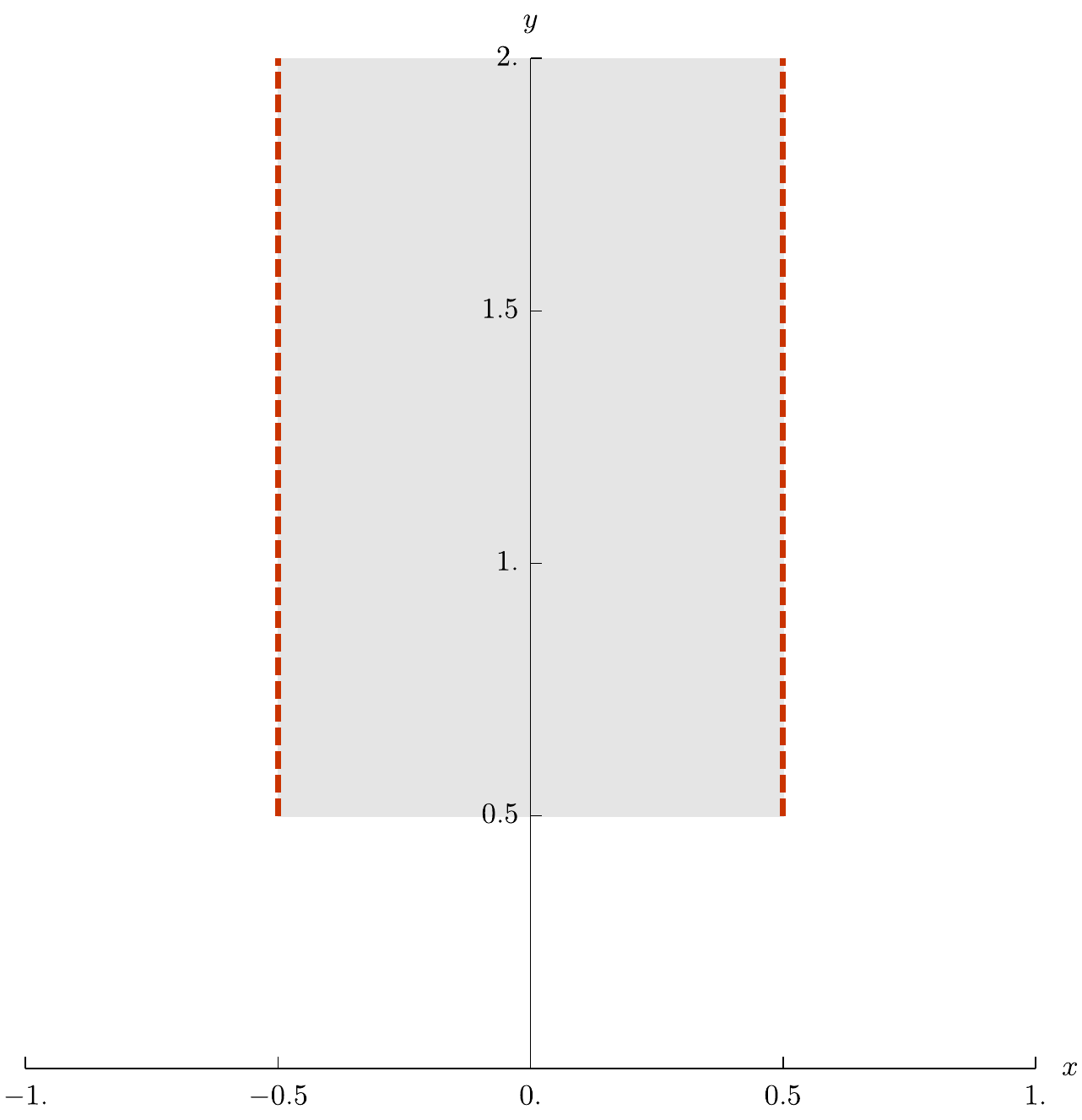}
 \caption[Fundamental domain of a cusp in $\mathbb{H}$]{Fundamental domain of a cusp in $\mathbb{H}$. The two parallel sides are identified}
 \label{fig1}
 \end{center}
\end{figure}

\begin{figure}[htb!]
 \begin{center} 
 \includegraphics{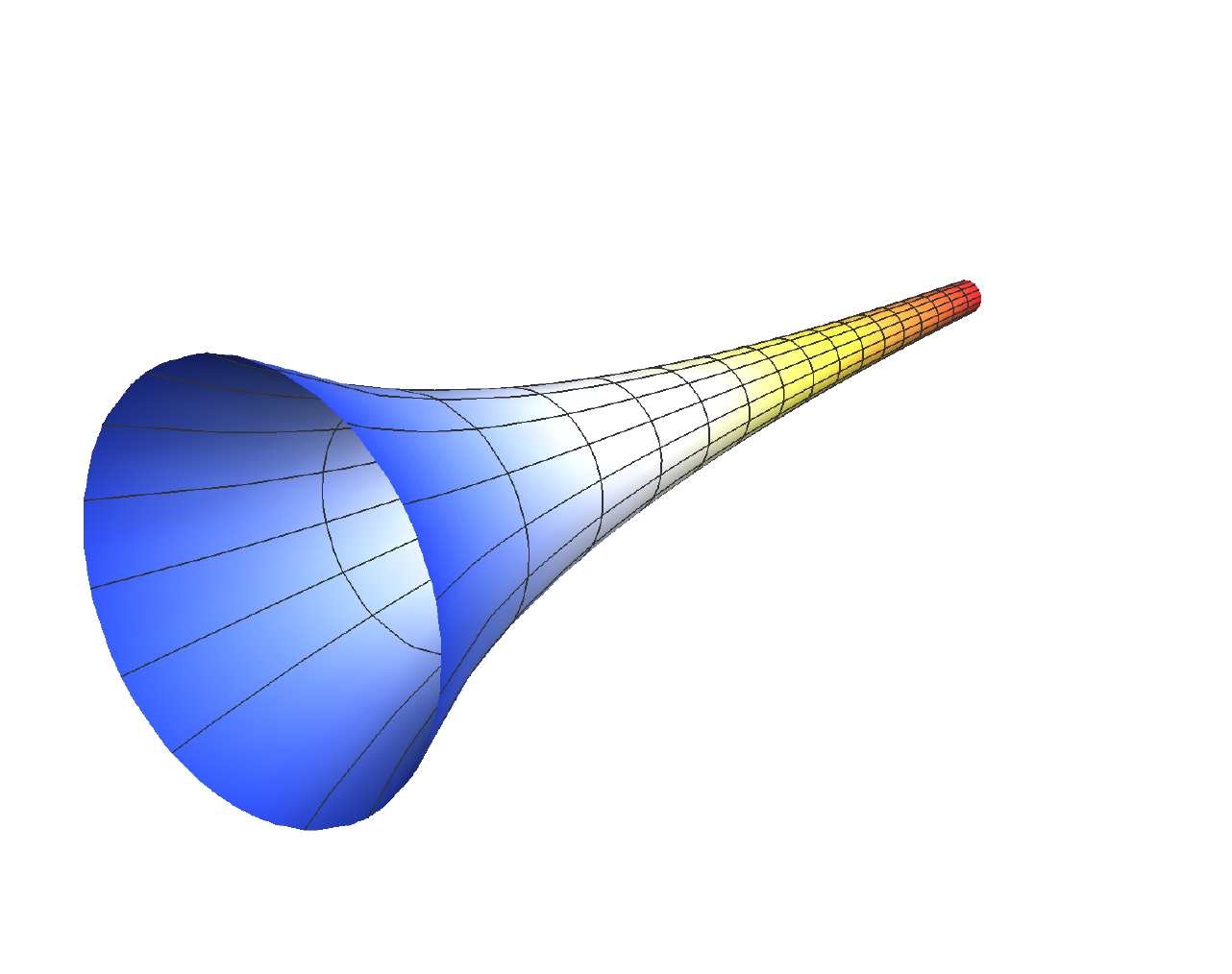}
 \caption{Part of a cusp isometrically embedded into $\mathbb{R}^3$}
 \label{fig2}
 \end{center}
\end{figure}

Figure \ref{fig1} shows a fundamental domain that becomes $Z^a$ when the parallel sides are identified.
Of course the space of smooth functions $C^\infty(Z^a)$ on $Z^a$ can be identified with
smooth functions on $\mathbb{H}$ periodic in $x$ (with period one) and similarly, $L^2(Z^a)$ can be identified with the set of measurable
functions $f(z)$ on $\mathbb{H}$, periodic in $x$ (with period one),  such that the $L^2$-norm
\[
 \int_a^\infty \int_{-1/2}^{+1/2} | f(z) |^2 y^{-2} \dr x \dr y
\]
is finite.
We will in the following use these identifications without further mention.
The Neumann Laplace operator $\Delta_{Z^a}$ on the cusp $Z^a$ is obtained by imposing Neumann boundary condition
on the operator $\Delta= -y^2 \left( \partial_x^2 + \partial_y^2 \right)$ on the boundary 
$(\RZ)\times \{a\}$. This operator is self-adjoint and has spectrum consisting of an absolutely continuous part
$[\frac{1}{4},\infty)$ and of a discrete set of non-negative eigenvalues with finite dimensional eigenspaces
(see for example \cite{MR725778} or \cite{MR1942691} and references there).

Suppose that $X$ is a complete two-dimensional Riemannian manifold (or orbifold with finitely many isolated orbifold singularities)
that is either a hyperbolic cusp or a disjoint union of hyperbolic cusps outside of a compact region.
Thus, we are assuming that $X$ has the form
\[
 X = M \cup_{\partial Z} Z,\quad Z=Z_1 \sqcup \cdots \sqcup Z_p, \quad Z_k = (\RZ) \times [a_k,\infty)
\]
such that the Riemannian metric $g$ on $X$ restricted to a neighbourhood of the cusp $Z_k$ is the hyperbolic
metric defined above
(see Figure \ref{fig3}). 

\begin{figure}[htb!]
\begin{center} 
\includegraphics{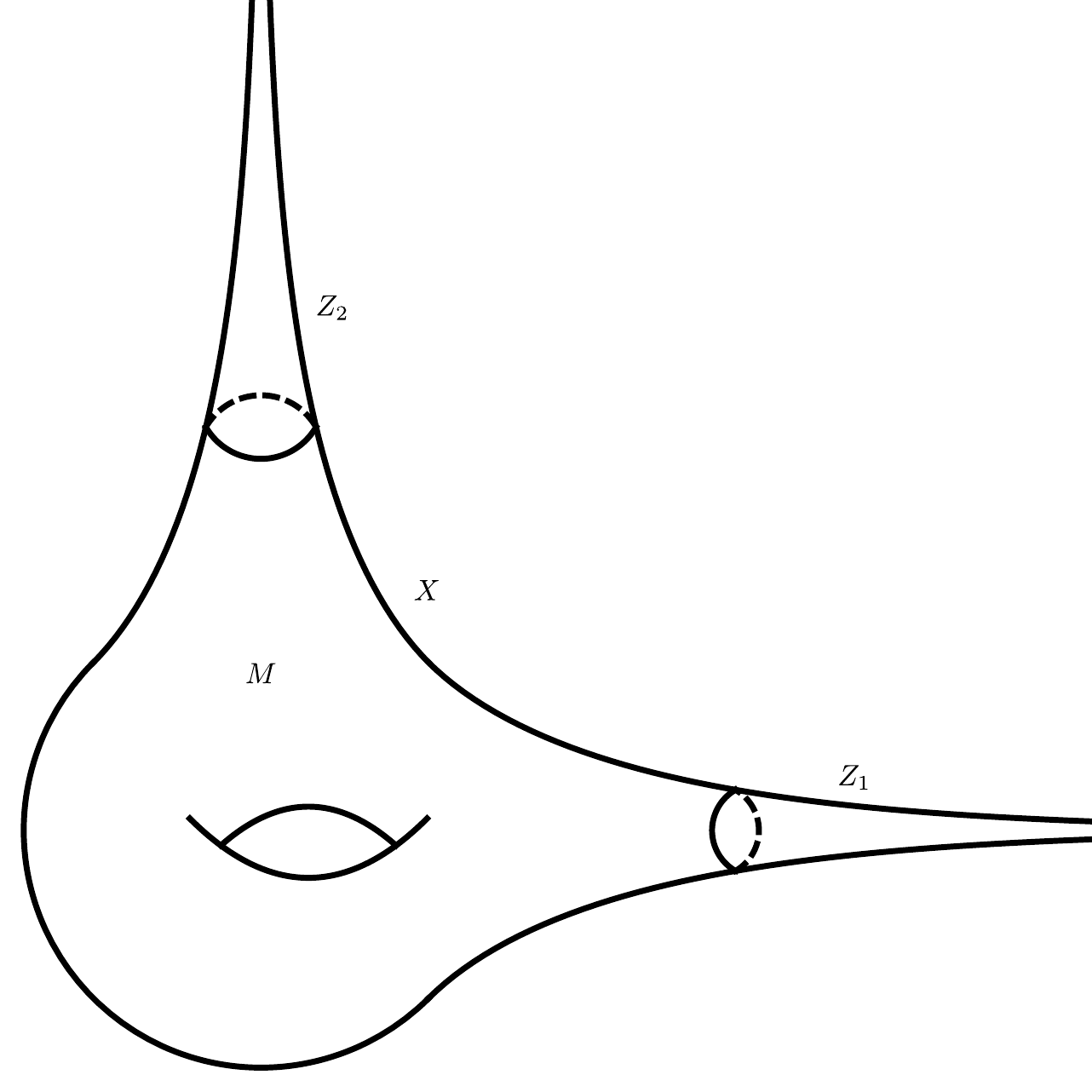}
\caption{A surface of genus one with two hyperbolic cusps}
\label{fig3}
\end{center}
\end{figure}

Now assume that $P$ is a formally self-adjoint differential operator of Laplace type on $X$ acting on functions (which means that $P=-g^{ij} \partial_i \partial_j + \text {lower order terms}$), and 
let $\Delta$ be the Laplace operator acting on functions on $X$. 
Thus, $P-\Delta$ is a first order operator and we will assume that $P-\Delta$ is compactly supported away from each $Z_k$.
The simplest example would be
\[
 P = \Delta + V(x),
\]
where $V \in C^\infty(X)$ is a potential that is supported in the interior of $M$. However, we do not want to exclude more
general cases here. Since  $X$ is complete the operator 
$P$ is essentially self-adjoint on $C^\infty_0(X)$. 

\begin{remark}All our formulae and conclusions hold true with the obvious
modifications if $M$ has additional boundary components and/or conical singularities away from cusps $Z$,  and appropriate elliptic boundary conditions are imposed there to make $P$ a self-adjoint operator.
\end{remark}

Manifolds with such cusps were considered and analysed in 
\cite{MR725778} and \cite{muller1992spectral} and the gluing constructions for the heat kernel carry over to our setting.
The structure of the spectrum and the generalised eigenfunctions can also be inferred from the meromorphic continuation of the
resolvent. This approach can be found for example in \cite{guillope1995upper}. In the following we summarise the known results.

As the Neumann Laplace operator on $Z^a$, the operator $P$
has spectrum consisting of the absolutely
continuous part $[1/4,\infty)$ of multiplicity $p$ and, maybe, eigenvalues of finite multiplicity.
As usual the resolvent $(P-\lambda)^{-1}$ is often more conveniently described using other parameters $s$ and $t$
which are related to the spectral parameter $\lambda$ in the following way,
\begin{gather*}
 \lambda = s (1-s),\\
 s = \frac{1}{2} + \I \; t,\\
 \lambda = \frac{1}{4} + t^2.
\end{gather*}
The set on which the resolvent is naturally defined as a meromorphic function with values in the space of bounded operators is 
$\mathbb{C} \backslash [\frac{1}{4},\infty)$ in terms of $\lambda$, the half-plane $\Re(s)>1/2$ in terms of $s$, 
and the lower half-space in terms of the parameter $t$.
The resolvent,
\[
 (P-\lambda)^{-1}=(P - s(1-s))^{-1},
\]
viewed as an operator from $L^2_\text{comp}(X)$ to $H^2_\text{loc}(X)$, admits a meromorphic continuation as a function of $s$ 
to the entire complex plane with poles of finite rank (that is, all the negative Laurent coefficients are finite rank operators).  These poles correspond to eigenvalues and so-called scattering
resonances. 

The generalised eigenfunction $E_j(z,s)$ of the operator 
$P$, attached to the cusp $Z_j$, can be constructed from the resolvent
and therefore admits a meromorphic continuation to $\mathbb{C}$ as a function of $s$.
When restricted to $Z_k$ it is of the form
\begin{gather} \label{ge-expansion}
 \left.E_j(z,s)\right|_{Z_k} = \delta_{j,k} y_k^s + C_{j,k}(s) y_k^{1-s} + T_{j,k}(z_k,s),
\end{gather}
where $T_{j,k}(z,s)$ is in $L^2(Z_k)$. Here $z_k = x_k + \I\, y_k$ denotes the coordinates on the cusp $Z_k$.
Both $C(s)$ and $T(z,s)$ are meromorphic matrix-valued functions of $s$
in the entire complex plane. The matrix-valued meromorphic function $C(s)$ is defined by \eqref{ge-expansion} and is normally referred to as
the scattering matrix. It satisfies the relations
\begin{gather}
 \overline{C(s)} =  C^*(s) = C(\overline s),\label{unitary}\\ 
 C(s) C(1-s) =\id, \label{functionaleq}
\end{gather}
implying that it is unitary on the absolutely continuous spectrum.

Since $Z^a$ has a natural $\bbS^1$-action we can decompose, in the case $P=\Delta$, the solutions of 
\[
(\Delta - s(1-s)) f(z) =0
\]
into the Fourier modes $f(z) = \sum\limits_{m \in \mathbb{Z}} f_m(y) e^{2 \pi \I m x}$ that satisfy
\[
 \left( - y^2 \frac{d^2}{dy^2} + 4 \pi^2 m^2 y^2 - s(1-s) \right) f_m(y) =0.
\]
For $m=0$ and $s\ne\frac{1}{2}$ this implies that $f_0(y)$ is a linear combination of $y^s$ and $y^{1-s}$.
For $m \not= 0$ the general solution of this ODE can be expressed in terms of Bessel functions.
Then we obtain
\[
 T_{j,k}(z,s) = \sqrt{y} \sum_{m \in \Z \backslash \{0\}} a_{m,j,k}(s) K_{\I t}(2 \pi |m| y) e^{2 \pi \I m x},
\]
where convergence is in $C^\infty(Z^a)$. Here $K_\nu$ is the modified Bessel $K$-function of order $\nu$.

Poles of the scattering matrix $C(s)$ are called resonances. Resonances correspond to poles
of the generalised eigenfunctions $E_j(z,s)$ and the coefficient of the lowest term in the Laurent expansion
at a resonance $s_r$ is proportional to a function $f_r \in C^\infty(X)$
such that
\[
 \left( P- s_r(1-s_r) \right) f_r =0,
\]
and
\[
 \left.f_r \right|_{Z_k} = a_{r,k} y_k^{1-s_r} + R_{r,k}(z_k),
\]
where $R_{r,k}$ is exponentially decaying as $y_k \to \infty$.
The function $f_r$ is sometimes referred to as the resonant state at the resonance $s_r$.

\section{Plan of the paper and discussion of the results}

The main aim of this paper is to demonstrate that the domain decomposition using the Neumann-to-Dirichlet map  leads to a simple  and fast  numerical scheme allowing to compute the scattering matrix on spaces with hyperbolic cusps. 

The paper is structured as follows.

In Sections \ref{sec:NDM} and \ref{sec:NDZ} we construct the Neumann-to-Dirichlet maps on the compact part of a hyperbolic surface and on the cusps, respectively.   
In Section \ref{ndscatter} we show that the scattering matrix can be extracted from the Neumann-to-Dirichlet operator of a compact part of a hyperbolic surface with cusps by means of simple linear algebra methods. In particular, if a numerical approximation of the Neumann-to-Dirichlet map at a spectral point is provided by any method, fast and standard linear algebra routines can be used to extract the scattering matrix. In Section \ref{numsec} we show that in fact standard finite element methods  are already sufficient to calculate the scattering matrix, and hence the scattering resonances, with good accuracy if the spectral parameter is not too large. Various examples of constant but also non-constant curvature are treated and discussed in detail in Section \ref{examples}. We compare them to known values for arithmetic surfaces as computed for example by Winkler (\cite{winkler1988cusp}) and Hejhal (\cite{hejhal92}).
       Since our method is extremely fast and flexible we were able to produce moving pictures that show how scattering resonances move 
       with conformal perturbations or in Teichm\"uller space. Figures illustrating this are included in Section \ref{examples}. In particular, in genus one case we identified several surfaces for which the scattering matrix is expressible in terms of the Riemann zeta function.       
       These surfaces correspond to the four arithmetic surfaces known to exist in genus one with one cusp. It seems that these arithmetic surfaces are the only ones (up to isomorphism) for which the resonances are lined up along critical lines.
              
       For surfaces of constant negative curvature there are direct fast converging methods that allow the computation of embedded eigenvalues and scattering resonances. For example Hejhal's algorithm can be used to compute embedded eigenvalues with extreme accuracy (see for example  \cite{MR2249995}, see also \cite{BSVdata}), and is used  to compute large numbers of high lying eigenvalues (for example \cite{JST13} for arithmetic examples).
       Variations have also been used to track resonances (for example \cite{MR2164106,MR2261026, MR2731549}). Our approach is different in that it treats the compact part as a black-box and also allows for perturbations away from constant curvature. The correspondence between the scattering matrix and the Neumann-to-Dirichlet map can be used to relate number theoretic questions to transmission problems. This approach was taken independently in \cite{CakoniChanillo} in the context of quotients of hyperbolic space by Fuchsian groups and leads to a  reformulation of the Riemann hypothesis in terms of transmission eigenvalues.
       
We would like to point out that numerical instabilities leading to spurious eigenvalues or eigenvalues being missed seem to be absent in our approach.
       We give several tables comparing our results to known computations in arithmetic constant curvature situations.

\section{The Neumann-to-Dirichlet operator on $\partial M$}\label{sec:NDM}

The operator $P$ is a formally self-adjoint elliptic differential operator on $X$ that coincides with the Laplace operator
near the boundary of $M$. 
Therefore, we have Green's formula
\[
 \langle P \psi, \phi \rangle - \langle \psi, P \phi \rangle = 
 \int_{\partial M} \left( \psi(z)  \frac{\partial \overline\phi}{\partial \bf n}(z) - 
  \frac{\partial \psi}{\partial \bf n}(z) \overline \phi(z) \right) \dr z
\]
for all $\phi,\psi \in C^\infty(M)$. 
In our case the boundary $\partial M$ is a disjoint union of components $\partial M_k=\partial Z_k$ each of which
is isometric to the circle. We therefore have
\[
 \int_{\partial M} \left( \psi(z)  \frac{\partial \overline \phi}{\partial \bf n}(z) - 
  \frac{\partial \psi}{\partial \bf n}(z) \overline \phi(z) \right) \dr z = \sum_{k=1}^p  \int_{\partial M_k} \left( \psi(z)  \frac{\partial \overline \phi}{\partial \bf n}(z) - 
  \frac{\partial \psi}{\partial \bf n}(z) \overline \phi(z) \right) \dr z.
\]
Given a particular boundary component $\partial M_k$ we can choose coordinates $(x,y)$ such that
the cusp $Z_k$ corresponds to $S^1 \times [a_k,\infty)$. In this case
$\frac{1}{a_k}dx$ is the natural Riemannian measure induced
by the metric on the boundary and $a_k \frac{\partial}{\partial y}$ is the unit normal vector field. We therefore have
\[
 \int_{\partial M_k} \left( \psi(z)  \frac{\partial \overline \phi}{\partial \bf n}(z) - 
  \frac{\partial \psi}{\partial \bf n}(z) \overline \phi(z) \right) \dr z=  \int_{-\frac{1}{2}}^{\frac{1}{2}} \left.\left( \psi(z)  a_k \frac{\partial \overline \phi}{\partial y}(z) - 
  a_k \frac{\partial \psi}{\partial y}(z) \overline \phi(z) \right)\right|_{y=a_k} \frac{1}{a_k}\, \dr x.
\]
We can hence construct another self-adjoint operator $P_\text{Neu}$ on $L^2(M)$ by restricting $P$ to $M$ 
and imposing Neumann boundary conditions at the boundary $\partial M$.
Since $P_\text{Neu}$ is self-adjoint and elliptic
there exists an orthonormal basis in $L^2(M)$ consisting of smooth eigenfunctions
$(\Phi_j)_{j \in \mathbb{N}}$ such that 
\begin{gather*}
 P_\text{Neu} \Phi_j = \lambda_j \Phi_j,\\
 \left. \frac{\partial \Phi_j}{\partial \bf n}\right|_{\partial M} = 0,
 \end{gather*}
where $ \lambda_1 \leq \lambda_2 \leq \ldots \to \infty$ are the corresponding eigenvalues.

If $\lambda\in\mathbb{C}$ is not a Neumann eigenvalue then for each $f \in C^\infty(\partial M)$ there exists a unique
function $\psi \in C^\infty(M)$ such that
\begin{equation}\label{eq:nonhomNeu}
\begin{split}
(P - \lambda) \psi &= 0,\quad\text{ in } M,\\
\left. \frac{\partial \psi}{\partial \bf n}\right|_{\partial M}  &= f.
\end{split}
\end{equation}
The so-called Neumann-to-Dirichlet operator $\mathcal{N}^{M}(s) : C^\infty(\partial M) \to C^\infty(\partial M)$ is defined as
\[
 \mathcal{N}^{M}(s) f := \psi \vert_{\partial M},
\]
where $\psi \in C^\infty(M)$ is the solution of \eqref{eq:nonhomNeu}.
 
Separating between the different boundary components the Neumann-to-Dirichlet map can also be thought of as a matrix
of operators
$\mathcal{N}^{M}_{kj}(s) : C^\infty(\partial M_j) \to C^\infty(\partial M_k)$.
It is well known that $\mathcal{N}^M(s)$ is a pseudodifferential operator
of order $-1$ whose full symbol depends only on the germ of the metric near the boundary  (see \cite{MR1029119} in case $s=0$, but the proof given there works in general).
In particular the off-diagonal terms of the matrix $\mathcal{N}^{M}_{kj}(s)$ are smoothing operators
and the diagonal ones are pseudodifferential operators of order $-1$ acting on $C^\infty(\partial M_j)$.
Using Green's formula one easily obtains
\begin{equation}\label{eq:ND}
\mathcal{N}^M(s) f = \sum_{j} \frac{1}{\lambda_j - s(1-s)} \langle f, \phi_j \rangle_{L^2(\partial M)}\; \phi_j,
\end{equation}
where $\phi_j = \Phi_j|_{\partial M}$ are the restrictions of the Neumann eigenfunctions $\Phi_j$ to the boundary $\partial M$ of $M$
and the sum converges in $H^{1/2}(\partial M)$ (see \cite{Levitin2008}). Taking differences one obtains
\begin{equation}\label{eq:NDdiff}
\left( \mathcal{N}^M(s) -\mathcal{N}^M(s_0) \right) f = \sum_{j} \frac{s_0(1-s_0)- s(1-s)}{(\lambda_j - s(1-s))(\lambda_j - s_0(1-s_0))}\langle f, \phi_j \rangle_{L^2(\partial M)}\; \phi_j.
\end{equation}
This converges in $H^{3/2}(\partial M)$ uniformly with respect to the $H^2(\partial M)$-norm of $f$.
In particular, $\mathcal{N}^M(s)$ is a meromorphic family of pseudodifferential operators
of order $-1$ with first order poles at $s_j$  that are related to the Neumann eigenvalues $\lambda_j$ of $P_\text{Neu}$ by $\lambda_j = s_j(1-s_j)$.  
The family of operators $\mathcal{N}^M(s) $ is hence completely determined
by the data $(\phi_j, \lambda_j)_{j \in \mathbb{N}}$.

\section{The Neumann-to-Dirichlet operator on cusps}\label{sec:NDZ}

Since the $Z^a$ admits an $S^1$-action the space $L^2(Z^a)$ every function $f \in L^2(Z^a)$
may be decomposed into Fourier modes
\[
 f(z) = \sum_{m \in \mathbb{Z}} f_m(y) e_m(x),
\]
where $e_m(x) = e^{2 \pi \I  m x}$.
The functions with vanishing zero Fourier coefficients form a sub-space in $L^2(Z^a)$, the so called
cuspidal functions
\[
 L^2_{\mathrm{cusp}}(Z^a) = \{f \in L^2(Z^a) \mid f_0(y) = 0 \; \mathrm{a.e.}\}.
\]
The orthogonal complement $L^2_0(Z^a)$ of $L^2_{\mathrm{cusp}}(Z^a)$ is then the space of functions that do not depend on $x$.
This space is canonically isomorphic to $L^2((a,\infty), y^{-2} dz)$. The Neumann Laplace operator leaves both spaces invariant. Its restriction
to $L^2_0(Z^a)$ has absolutely continuous spectrum $[\frac{1}{4},\infty)$ and the restriction to $L^2_{\mathrm{cusp}}(Z^a)$
has purely discrete spectrum consisting of eigenvalues of finite multiplicity accumulating at $\infty$.
If $\lambda=s(1-s)$ is not a eigenvalue of the Neumann Laplace operator on $L^2_{\mathrm{cusp}}(Z^a)$ then for each
$f \in L^2(S^1)$ with $\int\limits_{-1/2}^{1/2} f(x) \dr x =0$ there exists a unique function $\psi\in L^2(Z^a)$ such that
\begin{gather*}
 (\Delta - \lambda) \psi =0,\\
\left. -a \frac{\partial \psi }{\partial y}\right|_{y=a} = f.
\end{gather*}
We will define the cuspidal Neumann-to-Dirichlet operator $\mathcal{N}^{Z^a}(s) : C^\infty(S^1) \to C^\infty(S^1)$ as
\[
  \mathcal{N}^{Z^a}(s) ( f - \av(f)) := \psi|_{y=a},
\]
where 
\[
\av(f) := \int_{-1/2}^{1/2} f(x) \dr x.
\]
This operator has an explicit description in terms of Bessel functions. Namely, it follows directly from the expansion into Fourier modes that
for any $m \not=0$ we have
\[
   \mathcal{N}^{Z_k}(s) e_m = -\left( \frac{1}{2} + 2 \pi |m| a \frac{K'_{\I t}}{K_{\I t}}(2 \pi |m| a_k) \right)^{-1} e_m,
\]
and $\mathcal{N}^{Z_k}(s) e_0=0$.
Since the boundary of $M$ consists of a disjoint union of components $\partial M_k$ we can assemble the Neumann-to-Dirichlet operator to an operator
$\mathcal{N}^{c}(s)$ acting on $L^2(\partial M)= \bigoplus_{k=1}^p L^2(\partial M_k)$ by defining
\[
 \mathcal{N}^{c}(s) = \bigoplus_{k=1}^p  \mathcal{N}^{Z_k}(s) .
\]
In the same way the averaging operator can be assembled to a map $\av\!: L^2(\partial M) \to L^2(\partial M)$.

\section{The relation between the Neumann-to-Dirichlet operator and the scattering matrix}\label{ndscatter}

The generalised eigenfunctions $E_j(z,s) $ form a meromorphic family of functions
satisfying $( P- \lambda) E_j(z,s) =0$ on all of $M$. Hence, 
\[
\left.\mathcal{N}^{M}_{k l}(s) \left( a_l \frac{\partial}{\partial y_l} E_j(z_l,s)  \right) \right|_{\partial M_l}  = \left.E_j(z_k,s) \right|_{\partial M_k}.
\]
On the other hand the restriction of $E_j(z,s)$ to each cusp has an expansion of the form (\ref{ge-expansion}) with a decaying tail term. We therefore have
\[
\begin{split}
 &\mathcal{N}^{Z_k}(s) \left(\left.  - a_k \frac{\partial}{\partial y_k} ( E_j(z_k,s) - \delta_{j,k} y_k^s - C_{j,k}(s) y_k^{1-s})\right|_{\partial M_k}  \right) \\
 &\qquad= \left.\left(E_j(z_k,s) -  \delta_{j,k} y_k^s - 
 C_{j,k} (s) y_k^{1-s}\right)\right|_{\partial M_k}.
 \end{split}
\]
 This means in particular that $ \frac{\partial}{\partial \mathbf{n}} E_j(z,s) |_{\partial M}$  is in the kernel of the map  
\[
(\id - \av) \mathcal{N}^{M} + \mathcal{N}^{c}.
\]
 
 Note that the averaging map ${\av}\!: L^2(\partial M) \to L^2(\partial M)$ is the orthogonal projection onto the space of locally constant functions $L^2_0(\partial M)$ on $\partial M$.
 This space is naturally identified with $\mathbb{C}^p$, the $k$-th component being identified with the function value on the boundary component $\partial M_k$.

\begin{theorem}\label{main1}
 Suppose that $s \not=\frac{1}{2}$ is a complex number that is not a pole of $\mathcal{N}^{M}(s)$ or $\mathcal{N}^{c}(s)$, and not a pole
 of the scattering matrix $C(s)$. Suppose furthermore that $s(1-s)$ is not an $L^2$-eigenvalue of $P$. Then
 the kernel of the map $(\id - \av) \mathcal{N}^{M}(s) + \mathcal{N}^{c}(s)$ is $p$-dimensional and spanned by
 $\{ \frac{\partial}{\partial \mathbf{n}} E_k(z,s)|_{\partial M} \mid 1 \leq k \leq p\}$. \end{theorem}

\begin{proof}
 The assumptions imply that the generalised eigenfunctions $E_j(z,s)$ exist at $s$.
 We have already shown that $ \frac{\partial}{\partial \mathbf{n}} E_j(z,s) |_{\partial M}$ is in the kernel of $(\id - \av) \mathcal{N}^{M}(s) + \mathcal{N}^{c}(s)$. Moreover, any non-zero linear combination of $E=\sum_j c_j E_j$ such that $\frac{\partial}{\partial \mathbf{n}} E_j(z,s) |_{\partial M}=0$ will give rise to an $L^2$-Neumann eigenfunction on $Z$ by taking the non-zero part of its Fourier expansion. Since we excluded Neumann eigenvalues on the cusp by the requirement that 
 $s$ is not a pole of $\mathcal{N}^c$,  the functions $ \frac{\partial}{\partial \mathbf{n}} E_j(z,s) |_{\partial M}$ are linearly independent.
Now suppose that $g \in L^2(\partial M)$ is in the kernel of $(\id - \av) \mathcal{N}^{M} + \mathcal{N}^{c}$.  Both $(\id - \av) \mathcal{N}^{M}$ and 
 $\mathcal{N}^{c}$ are elliptic pseudodifferential operators of order $-1$ and their principal symbols coincide. Hence, their sum is elliptic too and, by elliptic regularity,
 $g \in C^\infty(\partial M)$. 
This means that there is a function $F_M \in C^\infty(M)$ and a function $F_Z \in L^2(Z) \cap C^\infty(M)$
 such that
 \begin{gather*}
  (P-s(1-s)) F_M = 0, \quad  (P-s(1-s)) F_Z = 0,\\
  (\id - \av) F_M |_{\partial M}= F_Z|_{\partial M}, \quad \left.(\id - \av) \frac{\partial}{\partial \mathbf{n}} F_M \right|_{\partial M}= \left.-\frac{\partial}{\partial \mathbf{n}} F_Z\right|_{\partial M},\\
 \left.\frac{\partial}{\partial \mathbf{n}} F_M \right|_{\partial M}=g.
 \end{gather*}
 These equations imply that the functions $F_M$ and $F_Z$, when expanded into Fourier modes, have the same non-zero Fourier coefficients on each cusp.
 Hence,  for each cusp $Z_k$ there exist coefficients $a_k$ and $b_k$ such that the function
\[
  a_k y_k^s + b_k y_k^{1-s} + F_{Z_k}
\]
 has the same Fourier expansion as $F_M$ on $Z_k$. Therefore, we can construct a globally defined function $F \in C^\infty(X)$
 which agrees with $F_M$ on $M$, such that
\[
  F|_{Z_k}(z_k) =  a_k y_k^s + b_k y_k^{1-s} + F_{Z_k}(z_k).
\] 
 Now use Green's identity on a cut-off domain $M_R$ obtained by cutting off the cusps $Z_k$ at $y=R$ and use the fact that the tail term is exponentially decaying:
 \[
 \begin{split}
  0 &= \lim_{R \to \infty} \int_{M_R} \left( E_j(z,s) ( P- s(1-s) ) F(z) - F(z) ( P- s(1-s) ) E_j(z,s) \right) \dr z \\
  &= (1-s) b_j + s \sum_{k=1}^p a_k C_{j,k}(s) - (1-s) \sum_{k=1}^p a_k C_{j,k}(s) - s b_j\\
  &= (1-2 s) \left( b_j - \sum_{k=1}^p a_k C_{j,k}(s) \right).
 \end{split}
 \]
 Define $E(z):= \sum\limits_{k=1}^p a_k E_k(z,s)$. Then the above implies that $F-E \in L^2$. Since $s(1-s)$ was assumed not to be an $L^2$-eigenvalue
 we conclude that $F=E$.
  \end{proof}
 
 \begin{theorem} \label{main2}
 Suppose that $s \not=\frac{1}{2}$ is a complex number that is not a pole of $\mathcal{N}^{M}(s)$ or $\mathcal{N}^{c}(s)$, and not a pole
 of the scattering matrix $C(s)$. Suppose furthermore that $s(1-s)$ is not an $L^2$-eigenvalue of $P$. 
 Let $V$ be the kernel of the map $(\id - \av) \mathcal{N}^{M}(s) + \mathcal{N}^{c}(s)$ and define the  maps
 \[
 \begin{alignedat}{2}
  Q_1: V &\to \mathbb{C}^p, &\quad &g \mapsto \av(g),\\
  Q_2: V &\to \mathbb{C}^p, &\quad &g \mapsto \av(\mathcal{N}^M(g)).
 \end{alignedat}
 \]
 Then the map $(s-1) Q_2 + Q_1$ is invertible and 
\begin{equation}\label{eq:Cformula}
  C(s) = A^{s-1} (s Q_2 - Q_1) \left( (s-1) Q_2 + Q_1 \right)^{-1} A^{s},
\end{equation}
 where $A$ is the diagonal matrix $A=\operatorname{diag}(a_1,a_2,\ldots,a_p)$.
 \end{theorem}
 
 \begin{proof}
   By the previous theorem we can use the basis $\left\{ \phi_k := \frac{\partial}{\partial \mathbf{n}} E_k(z,s)|_{\partial M} \mid 1 \leq k \leq p\right\}$ in $V$ to check the
   invertibility of the map $(s-1) Q_2 + Q_1$ and the formula for the scattering matrix. By the expansion (\ref{ge-expansion}) we
   have
   \[
   \begin{split}
    (Q_1 \phi_j)_k &= s \delta_{j,k} a_k^{s} + (1-s) C_{j,k}(s) a_k^{1-s},\\
    (Q_2 \phi_j)_k &= \delta_{j,k} a_k^{s} + C_{j,k}(s) a_k^{1-s}.
   \end{split}
   \]
  Therefore, $\left(((s-1) Q_2 + Q_1) \phi_j \right)_k= (2s-1) \delta_{j,k}  a_k^{s}$ and the right hand side is a non-singular matrix. Moreover,
   $\left((s Q_2 - Q_1) \phi_j \right)_k= (2s-1) C_{j,k}(s)  a_k^{1-s}$. The formula \eqref{eq:Cformula} is immediately implied by this.
 \end{proof}
 
In the case of one cusp the above reduces to a generalised eigenvalue problem.

\begin{cor}
 Assume $X$ has one cusp, i.e. $p=1$, and suppose that $s \not=\frac{1}{2}$ is a complex number that is not a pole of $\mathcal{N}^{M}(s)$ or 
 $\mathcal{N}^{c}(s)$, and not a pole of the scattering matrix $C(s)$. Then either the pair  $(\mathcal{N}^{M}(s) + \mathcal{N}^{c}(s), \av)$  has precisely one
 generalised eigenvalue $G(s)$, or $C(s) = \frac{s}{s-1} a^{2s-1}$. In the former case the scattering matrix can be computed from this eigenvalue as
\[
  C(s) = (s G(s) -1) \left( (s-1) G(s) + 1 \right)^{-1} a^{1-2s}.
\]
 \end{cor}

\section{Numerical computation of the scattering matrix, resonances and embedded eigenvalues} \label{numsec}
\subsection{Scattering matrix and resonances}
Theorems \ref{main1} and \ref{main2} yield an extremely simple and fast algorithm to compute the scattering matrix, resonances or eigenvalues
for the situation described above.
In this section we will assume that $s \not=\frac{1}{2}$ is a 
complex number that is not a pole of $\mathcal{N}^{M}(s)$ or $\mathcal{N}^{c}(s)$, and not a pole of the scattering matrix $C(s)$.
In the following we take $(e_m)_{m \in \mathbb{Z}}$ to be the orthonormal basis of $L^2(\R/\Z,dx)$ consisting of Fourier modes $e_m(x) = e^{2 \pi \I m x}$.
Since each boundary component $\partial M_k$ can be identified with a circle, this gives an orthonormal basis $(e_{m,k})_{m \in \mathbb{Z}, \;k =1,\ldots, p}$
in $L^2(\partial M)$. We will write $(e_\alpha)_{\alpha\in I}$ where the index set for $\alpha=(\alpha_1,\alpha_2)$ is  $I:=\mathbb{Z} \times \{1,\ldots,p\}$.

The boundary data $(\phi_j, \lambda_j)_{j \in \mathbb{N}}$ of Neumann eigenvalues can be used to compute the matrix elements of
the Neumann-to-Dirichlet operator $\mathcal{N}^M(s)$ using \eqref{eq:ND} and the Fourier expansion in the basis  $(e_\alpha)$ giving
\begin{equation}\label{eq:ND1}
 \mathcal{N}^{M}_{\alpha,\beta}(s) = \langle \mathcal{N}^M(s) e_\alpha , e_\beta \rangle_{L^2(\partial M)} = \sum_{j} \frac{1}{\lambda_j - s(1-s)} 
 \langle e_\alpha, \phi_j \rangle \langle \phi_j, e_\beta \rangle_{L^2(\partial M)}.
\end{equation}
Using \eqref{eq:NDdiff}, convergence in \eqref{eq:ND1} is accelerated if we compute the matrix elements  $\mathcal{N}^{M}_{\alpha \beta}(s_0)$ directly at a single particular value $s_0$, cf. \cite{Levitin2008}. Then
\begin{equation}\label{eq:ND2}
 \mathcal{N}^{M}_{\alpha,\beta}(s)- \mathcal{N}^{M}_{\alpha,\beta}(s_0)=\sum_{j} \frac{s_0(1-s_0)- s(1-s)}{(\lambda_j - s(1-s))(\lambda_j - s_0(1-s_0))}
 \langle e_\alpha, \phi_j \rangle \langle \phi_j, e_\beta \rangle_{L^2(\partial M)},
 \end{equation}
 and the series in \eqref{eq:ND2} converges more rapidly than the one in \eqref{eq:ND1}. The acceleration trick may be repeated if one computes directly $\mathcal{N}^{M}_{\alpha,\beta}(s_j)$ for several particular values of $s_j$.

The matrix elements of $\mathcal{N}^{Z^a}(s)$ are simply
\begin{equation}\label{eq:Ncusp}
 \mathcal{N}^{Z^a}_{\alpha,\beta}(s)  = 
 \begin{cases}
 0&\quad\text{if }\alpha\ne\beta,\\
 (1- \delta_{m,0}) \left( \frac{1}{2} + 2 \pi |m| a_k \frac{K'_{\I t}}{K_{\I t}}(2 \pi |m| a_k) \right)^{-1}&\quad\text{if }\alpha=\beta=(m,k).\
 \end{cases}
\end{equation}
Moreover, $\av_{\alpha,\beta} = \delta_{\alpha_1,0} \delta_{\beta_1,0}$. 

We would then like to find the $p \times p$ matrix $G(s)$ such that
\[
  \dim \mathrm{ker} \left(  \mathcal{N}^{M}(s) + \mathcal{N}^{c}(s) - G(s) \av \right)  = p.
\]
The idea of the numerical approximation is of course to truncate this Fourier basis and approximate the above matrices by finite matrices
by considering only $0 \leq |\alpha_1|, |\beta_1| \leq J$ for some large integer $J$.
We denote by $\tilde{N}^M$, $\tilde {N}^c$, and $\tilde{\av}$ the finite matrices obtained from truncating the
Fourier expansion at $J$. Then these matrices are $(2 J +1) p  \times (2 J +1)p$ matrices with complex entries.

We use the finite element method to compute the Neumann boundary data $(\phi_j, \lambda_j)_{j \in \mathbb{N}}$ in terms of the numbers
$\lambda_j$ and the Fourier modes $\langle \phi_j, e_\alpha \rangle$. Once these data are obtained a finite element approximation to 
$\tilde N^{M}(s)$ can be computed very quickly for arbitrary
$s \in \mathbb{C}$ in a given compact subset of the complex plane. 
The matrices $\tilde{N}^c$ can be computed very fast using a well known continued fraction expansion for the Bessel $K$-function \cite[Section 17]{Cuyt2008},
\begin{equation}\label{eq:Kexpansion}
 \left( \frac{1}{2 } + 2 \pi |m| a \frac{K'_{\I t}}{K_{\I t}}(2 \pi |m| a) \right) =  - 2 \pi |m| a - \mathop{\text{\Large$\mathbf{K}$}}_{n = 1}^\infty \left( \frac{-t^2- \frac{(2n -1)^2}{4}}{4 \pi |m| a + 2n} \right),
\end{equation}
where we use Gauss' notation
\[
\mathop{\text{\Large$\mathbf{K}$}}_{n = 1}^\infty \frac{p_n}{q_n}=\cfrac{p_1}{q_1+\cfrac{p_2}{q_2+\cfrac{p_3}{q_3+\cdots}}}.
\]

In order to compute the scattering matrix numerically in the above approximation we proceed as follows. By Theorem \ref{main1} the operator
$T(s) := (\id- \av) \mathcal{N}^M(s) +  \mathcal{N}^c(s)$ has a $p$-dimensional kernel spanned by $\frac{\partial}{\partial \mathbf{n}} E_k(z,s)|_{\partial M}$.
We compute the cut off approximation 
\[
 \tilde T(s) :=(\id- \tilde{\av}) \tilde{N}^M(s) +  \tilde{N}^c(s).
\]
If $J$ is large enough this matrix will have precisely $p$ small singular values. We can therefore perform a singular value decomposition to construct
an orthonormal system of singular vectors $(v_1,\ldots,v_p)$ with small singular values. The system of vectors $(\tilde{\av} (v_1),\ldots,\tilde{\av} (v_p))$
determines a $p \times p$ matrix $\tilde Q_1$. Similarly the vectors
$(\tilde{\av}\tilde{N}^M (v_1),\ldots,\tilde{\av} \tilde{N}^M (v_p))$ determine a $p \times p$ matrix $\tilde Q_2$. 
Since the set of invertible maps is open the matrix $(s-1) \tilde Q_2 + \tilde Q_1$ is invertible if the approximation is good enough. By Theorem \ref{main2} we then get a numerical approximation
of the scattering matrix by
\[
  \tilde C(s) = A^{s-1} (s \tilde Q_2 - \tilde Q_1) \left( (s-1) \tilde Q_2 + \tilde Q_1 \right)^{-1} A^{s}.
\]
 As before $A$ is the diagonal matrix $A=\mathrm{diag}(a_1,a_2,\ldots,a_p)$.

Since resonances are poles of $C(s)$ and we have the functional equation $C(s)C(1-s)=\id$, the resonances are precisely the zeros of the determinant of $C(1-s)$.

\subsection{Error estimates for the scattering matrix} \label{errorsection}

In this section we will show that in principle the error in the computation can be made rigorous if the exterior and interior Neumann-to-Dirichlet maps are obtained by a method with rigorous errors. Let us start assuming that we have a mechanism at our disposal to estimate the first Sobolev norm of
$\mathcal{N}^M(s) \Phi - \Psi$ for given smooth functions $\Phi$ and $\Psi$.
This depends on a chosen method of computation of the Neumann-to-Dirichlet map.

On $L^2(\partial M)$ we have the orthonormal basis $(e_{m,k})_{m \in \mathbb{Z}, \;k =1,\ldots, p}$. We define the Fourier multiplier 
$q: H^s(\partial M) \to H^{s-1}(\partial M)$ by $q e_{m,k} = \left( |m|+1 \right) e_{m,k}$. The operator $q$ is a first order pseudodifferential operator and can also be expressed in terms of the Laplace operator on the boundary. For concreteness we fix the $H^s$-norm on $\partial M$ as $\| \Phi \|_{H^s(\partial M)} = \| q^s \Phi\|_{L^2(\partial M)}$.

In the following we will assume that $s \in \mathbb{C}$ is fixed such that the assumptions of Theorem \ref{main2} hold:  $s \not=\frac{1}{2}$ is not a pole of $\mathcal{N}^{M}(s)$ or $\mathcal{N}^{c}(s)$, not a pole of the scattering matrix $C(s)$ and $s(1-s)$ is not an eigenvalue of $P$. 
Then
\[
 T'(s) := q T(s)=q\left( (\id- \av) \mathcal{N}^M(s) + \mathcal{N}^c(s) \right)
\]
is a zero order elliptic pseudodifferential operator. In particular $0$ is not in the essential spectrum of $T'^* T'$. This implies that the self-adjoint operator $|T'|$
has $0$ as a multiplicity $p$ eigenvalue and a spectral  gap in the sense that the spectrum is contained in $\{0\} \cup [K_1,\infty)$ for some $K_1>0$. 

Our numerical approximation takes place in the finite dimensional subspace $\mathcal{W}_J$ of functions $f$ that have a finite Fourier expansion of the form
\[
 f(z) = \sum_{| m | \leq J} f_m(y) e_m(x).
\]
As before $J$ is a sufficiently large integer.
The method will then usually find an orthonormal set  vectors $(v_1,\ldots,v_p)$ in $\mathcal{W}_J$ such that
\[
  \| T' v_k \| < \delta_1 \ll 1.
\]
If $P_0$ is the orthogonal projection onto the $p$-dimensional kernel of $|T'|$ it follows that
\[
 \| (\id - P_0) v_k \| \leq K_1^{-1} \delta_1.
\]
Applying the numerical approximation of $\mathcal{N}^M(s)$ we obtain another set of vectors $(w_1,\ldots,w_p)$ in the subspace. Given an error estimate on the Dirichlet-to-Neumann map as assumed we will get a bound of the form
\[
 \| \mathcal{N}^M(s) v_k -w_k \|_{L^2} < \delta_2.
\]
The  approximations $\tilde Q_1$ and $\tilde Q_2$ of the maps $Q_1$ and $Q_2$ can be though of as finite rank operators with range in the subspace 
$\mathcal{W}_J$ that vanish on the orthogonal complement of $\mathcal{W}_J$.
Recall that  $\tilde Q_1 =(\tilde\av(v_1),\ldots,\tilde\av(v_p))$ and $\tilde Q_2 =(\tilde\av(w_1),\ldots,\tilde\av(w_p))$. If we choose $(P_0 v_1, \ldots, P_0 v_p)$ as a basis
in the kernel of $T$ to describe $Q_1$ and $Q_2$ we obtain
\[
  \| Q_1 - \tilde Q_1  \| \leq \sqrt{p} \;  K_1^{-1}\delta_1  , \quad \| Q_2 - \tilde Q_2 \| \leq \sqrt{p} \left( \delta_2  + \|\mathcal{N}^M(s)\| K_1^{-1} \delta_1 \right),
\]
where the norms are the operator norms of the respective matrices. Let 
\[
\begin{split}
\epsilon_1 &=  \sqrt{p}K_1^{-1} \delta_1 +  | s-1| \sqrt{p} \left( \delta_2  + \|\mathcal{N}^M(s)\| K_1^{-1} \delta_1 \right),\\
 \epsilon_2 &=   \sqrt{p}K_1^{-1}\delta_1 +  | s | \sqrt{p} \left( \delta_2  + \|\mathcal{N}^M(s)\| K_1^{-1} \delta_1 \right) 
\end{split}
\]
and 
\[
 K_2 =\left\| \left( (s-1) \tilde Q_2 + \tilde Q_1\right)^{-1}\right \|, \quad K_3=\left \|  s \tilde Q_2 - \tilde Q_1 \right\|.
\]
Then, assuming $\epsilon_1 \, K_2 < 1$, we obtain
\[
 \left\| \left( (s-1) Q_2 + Q_1\right)^{-1} -  \left( (s-1) \tilde Q_2 + \tilde Q_1\right)^{-1} \right\| \leq \frac{\epsilon_1 K_2^2}{ 1- \epsilon_1 \,K_2 }.
\]
Collecting everything we can now estimate the error of the approximated scattering matrix
\[
  \tilde C(s) = A^{s-1} (s \tilde Q_2 - \tilde Q_1) \left( (s-1) \tilde Q_2 + \tilde Q_1 \right)^{-1} A^{s}.
\]
as
\[
\| \tilde C(s) - C(s) \| \leq \| A^{s-1} \| \| A^{s} \|  \left(  \frac{\epsilon_1 \, K_2^2 (K_3 + \epsilon_2) }{ 1- \epsilon_1 \,K_2}  +  \epsilon_2 \, K_2 \right).
\]

If the error for the Neumann-to-Dirichlet map is known, and the norm $\|\mathcal{N}^M(s)\|$ and the spectral gap $K_1$ can be estimated or computed, the error for the scattering matrix can be explicitly bounded. In principle this makes it possible to use interval arithmetics to rigorously prove interval bounds for the scattering matrix.
The scattering matrix is holomorphic in the resolvent set and by a classical theorem of Hurwitz the zeros of uniform approximations of $C(s)$ converge to zeros of $C(s)$. A quantitative version of this is given in 
\cite{R69}. This allows to estimate the error of the approximation of the computed resonances. 
Note that using a finite truncation of \eqref{eq:NDdiff} approximates the Neumann-to-Dirichlet operator in the correct norm. 

\subsection{Embedded eigenvalues}

Since we assumed $P$ was self-adjoint any eigenvalues will have to be on the real line. There are two classes of eigenvalues: those below the continuous spectrum
and those embedded into the continuous spectrum. We will refer to the eigenvalues $\lambda < 1/4$ as small eigenvalues and the eigenvalues $\lambda \ge 1/4$
as the embedded eigenvalues. Embedded eigenvalues correspond to real values of $t$ and therefore the real part of $s$ for these eigenvalues will always be $1/2$.
As a consequence the zero modes of the Fourier expansion of these eigenfunctions in the cusp has to vanish. We therefore make the following observation.

\begin{theorem}\label{main3}
 The embedded eigenvalues away from the poles of $\mathcal{N}^M$ and $\mathcal{N}^c$ are exactly those values of $\lambda=s(1-s)\in\left[\frac{1}{4},+\infty\right)$ for which there exists a non-zero vector $f \in C^\infty(S^1)$
 such that $\av(f)=0$ and 
\[
  \left( \mathcal{N}^M(s) + \mathcal{N}^c(s) \right) f =0.
\]
\end{theorem}

Computations of embedded eigenvalues face the problem that it is not possible to numerically distinguish between an embedded eigenvalue and a resonance that is close to the spectrum.  Rigorous error estimates that guarantee the existence of an embedded eigenvalue therefore always need some additional information about the geometry or, in the constant curvature case, arithmetic nature of the surface (see for example \cite{MR2249995}). The mathematically rigorous numerical part of this work is mostly about the computation of the scattering matrix and of resonances. We therefore only briefly sketch how one detects embedded eigenvalues or resonances close to the spectrum.
We are looking for vectors $v$ that satisfy
$\tilde{\av}(v) = 0$ and for which  $\left( \tilde N^M + \tilde N^{c} \right)v$ is small.
For numerical stability the $\mathrm{QR}$-decomposition $\tilde B(s)= \tilde{\mathbf Q}(s) \tilde{\mathbf R}(s)$ of the matrix
\[
\tilde B := \left( \tilde N^M + \tilde N^{c} \right) \oplus \tilde{\av} \oplus \tilde N^M \oplus \tilde N^c
\]
is performed. This matrix maps $\mathbb{C}^{(2 J +1) p}$ to  $\mathbb{C}^{(2 J +1) p} \oplus \mathbb{C}^{(2 J +1) p} \oplus \mathbb{C}^{(2 J +1) p} 
\oplus \mathbb{C}^{(2 J +1) p}$.
Let ${\mathbf P}$ be the projection onto the first two summands. We are looking to find values of $s$ for which there exists a vector $v$ for which 
${\mathbf P} \tilde B(s) v$ is very small whereas $(1-p) \tilde B(s) v$ is not. Since $\tilde{\mathbf R}$ is invertible these are exactly the small singular values of the matrix
${\mathbf P} \tilde{\mathbf Q}(s)$. Thus, our method of finding embedded eigenvalues is to plot the smallest singular value of ${\mathbf P} \tilde{\mathbf Q}(s)$ as a function of $s=\frac{1}{2} + \I \sqrt{\lambda -\frac{1}{4}}$. 
If the smallest singular value is close to zero for some $s=\frac{1}{2} + \I \sqrt{\lambda -\frac{1}{4}}$ this amounts to a small spectral gap $K_1$, i.e. a small $(p+1)$-st singular value of $T'$. Hence the error estimate for the scattering matrix near such a point becomes much worse, reflecting the fact that we may also have a resonance close to the spectrum.
The existence of resonances near the spectrum or embedded eigenvalues and completeness of a list of computed values can heuristically be verified using Turing's method and variants of the Weyl law that have been proved in this context \cite{MR725778}. Versions of the Weyl law with error estimates are available in the constant curvature case, see e.g. \cite{str19}. 

\section{Examples and numerical studies} \label{examples}

\subsection{General set-up}
In this section we will consider several examples of manifolds with cusps. We will be focusing on the Laplace operator acting on functions, i.e. 
in all the examples we will have $P=\Delta$. These examples are divided into groups as follows.
\begin{itemize}
 \item[$A_\phi$] the modular domain with its constant curvature metric changed by a conformal factor $\mathrm{e}^{\phi}$. This family is parametrised by smooth functions 
 $\phi$ on the modular surface.
 \item[$B_r$] a triangular domain that is sometimes referred to as Artin's billiard and that interpolates between the Hecke triangular surfaces. This family is parametrised by a real number  $r>\frac{1}{2}$.
 \item[$C_{\ell,\tau}$] the surfaces of genus $(1,1)$ and constant curvature, i.e.\ the punctured torus. The Teichm\"uller space of genus $(1,1)$ has dimension $2$ and therefore this family is parametrised by a length parameter $\ell>0$ and a twist parameter $\tau \in [0,1)$.
 \item[$D$] the unique hyperbolic surface of genus zero with three cusps.
\end{itemize}

In all these examples we decomposed the surface into compact part $M$ and a cusp-part.  The method allows the freedom of choosing a cut-off parameter $a$. In the examples below $a$ was usually chosen in the interval $[0.3, 2]$, depending on the geometry. Note that choosing significantly higher values of $a$ decreases the accuracy of a Neumann-to-Dirichlet map approximation, and choosing a small $a$ creates meshing problems due to a ``narrow'' compact part $M$.   Experiments indicate that the dependence of computed eigenvalues and resonances upon a choice of $a$ in a suitable subset of the above interval is negligible.

To compute a numerical approximation of 
the Neumann-to-Dirichlet map we use the accelerated expansion \eqref{eq:ND2}. We used the finite element framework
FreeFEM++ (\cite{hecht2012freefem} and \cite{hecht2012new}) to compute the Neumann-to-Dirichlet map at some point $s_0$ 
and to compute 
the boundary data of the first 1000 Neumann eigenvalues. On the boundary Fourier modes up to $|m|=40$ were used. In the FEM implementation we used discretisation with up    to 200 points on the boundary of the compact part. The  Neumann-to-Dirichlet map  on the cusps is computed using \eqref{eq:Ncusp} and \eqref{eq:Kexpansion}.

The corresponding data were expressed in terms of Fourier modes on the boundary and imported into a Mathematica 
script that directly computed the scattering matrix by the method described before. Since then the scattering matrix was available as a numerical function, 
we used Newton's root finding algorithm to locate zeros of its determinant. The functional equation \eqref{functionaleq} was then used to determine the scattering resonances. 
The poles and the zeros of the scattering matrix are located in the half-planes 
$\Re{s} > \frac{1}{2}$ and $\Re{s} < \frac{1}{2}$, respectively. Unless resonances are very close to the spectrum there are therefore no issues due to poles and zeros being close together. One can thus use the argument principle and contour integration to count the number of resonances in a region bounded away from the spectrum. We have found that in practice Newton's root finding algorithm finds all resonances away from the spectrum in a fast reliable manner. This is due to the well-behaved analytic properties of the scattering matrix. To locate and track resonances that are very close to the spectrum we start from a perturbation of the surface and then use predictive algorithms based on polynomial extrapolation to follow the path of the resonance. This way even the resonances that seem to have high order touching points with the spectrum could be tracked.

The following conventions are assumed in all videos and graphs:
\begin{itemize}
\item Resonances and eigenvalues are traced in a part of  $(\Re s, \Im s)\in\left(-\infty,\frac{1}{2}\right]\times[0,+\infty)$ quadrant of the $s$-plane.
\item The resonances very close to the continuous spectrum $\frac{1}{2}+\I\left[0,+\infty\right)$, as well as embedded eigenvalues, are shown in blue. Not all embedded eigenvalues are shown. 
\item The resonances on the critical line $\frac{1}{4}+\I[0,+\infty)$ or very close to it are shown in red.
\item  The resonances on the line $\Re s=0$ or very close to it are shown in green.
\item The eigenvalue at $s=0$ is never shown.
\item In graphs showing the trajectories of resonances, the starting points of the trajectories are marked by a disk, and the end points by a square.
\item In geometry figures, the arcs shown in the same colour and line type are identified; solid black lines indicate Neumann conditions imposed on the arcs.
\end{itemize}

\subsection{Benchmarking}
In the case of the modular surface $A_0$ the scattering matrix can be expressed in terms of the Riemann zeta function, see \eqref{eq:Cmodular}, and we could compare and compute the relative error of our approximation.
The scattering matrix $C(s)$ computed for $s = \frac{1}{2} + \I t$, $t \in [0.,30.]$ (this amounts to the interval $[0.25,900.]$ in the spectrum) had a maximal relative error of about $0.25\%$.
On the interval $[0,10]$ for $t$ we even obtained a maximal relative error not exceeding $0.004\%$. Similar errors hold on the critical line. We note that these approximations are surprisingly good considering that a finite element approximation was used. The finite element method and subsequent computations were carried out with double precision. 

As shown in Subsection \ref{errorsection} the $(p+1)$-th singular value of $ q \tilde T$ is a measure of the spectral gap. Away from resonances close to the real line or embedded eigenvalues the size of the first $p$ smallest singular values of $q \tilde T$ compared to the $(p+1)$-th small singular value was extremely small (typically of an order of a double precision rounding error) in our computations. Hence, using the terminology of Subsection \ref{errorsection} the numerical estimate for $K_1^{-1} \delta_1$ was very small and the theoretical error was dominated by $\delta_2$ which stems from the approximation of the Neumann-to-Dirichlet map. In our case most of the errors are due to the FEM approximation and decrease with mesh refinement.

The computational cost of building the scattering matrices using FEM realisation of Neumann-to-Dirichlet maps is relatively low if $\Im s\le 30$ (this of course depends on the FEM implementation and the number of eigenvalues used). The real runtime costs actually occur when we look for complex roots and poles of the scattering matrix and trace individual resonance dependence on the parameters.  

\subsection{$A_\phi$. The modular domain and conformal perturbations }

\subsubsection{Description of the surface}

This surface can be obtained from the domain
\[
\left\{ (x,y) \in \mathbb{H} \mid x^2+y^2 \geq 1, \; -\frac{1}{2} \leq x \leq \frac{1}{2} \right\}
\]
by gluing along the boundary as follows. The sides $x=-\frac{1}{2}$ and $x=\frac{1}{2}$
are identified by means of the parallel translation $x \mapsto x+1$. The circular arc
$\{ (x,y) \in \mathbb{H} \mid x^2+y^2 =1, 0 \leq x \leq \frac{1}{2} \}$ is identified  with
$\{ (x,y) \in \mathbb{H} \mid x^2+y^2 =1, -\frac{1}{2} \leq x \leq 0 \}$ using the map $x \mapsto -x$.
This results in a hyperbolic surface with one cusp and two orbifold singularities at the points $(0,1)$ and $(1/2,\sqrt{3}/2)$, the latter identified with $(-1/2,\sqrt{3}/2)$.

\begin{figure}[htb!]
\begin{center} 
\includegraphics{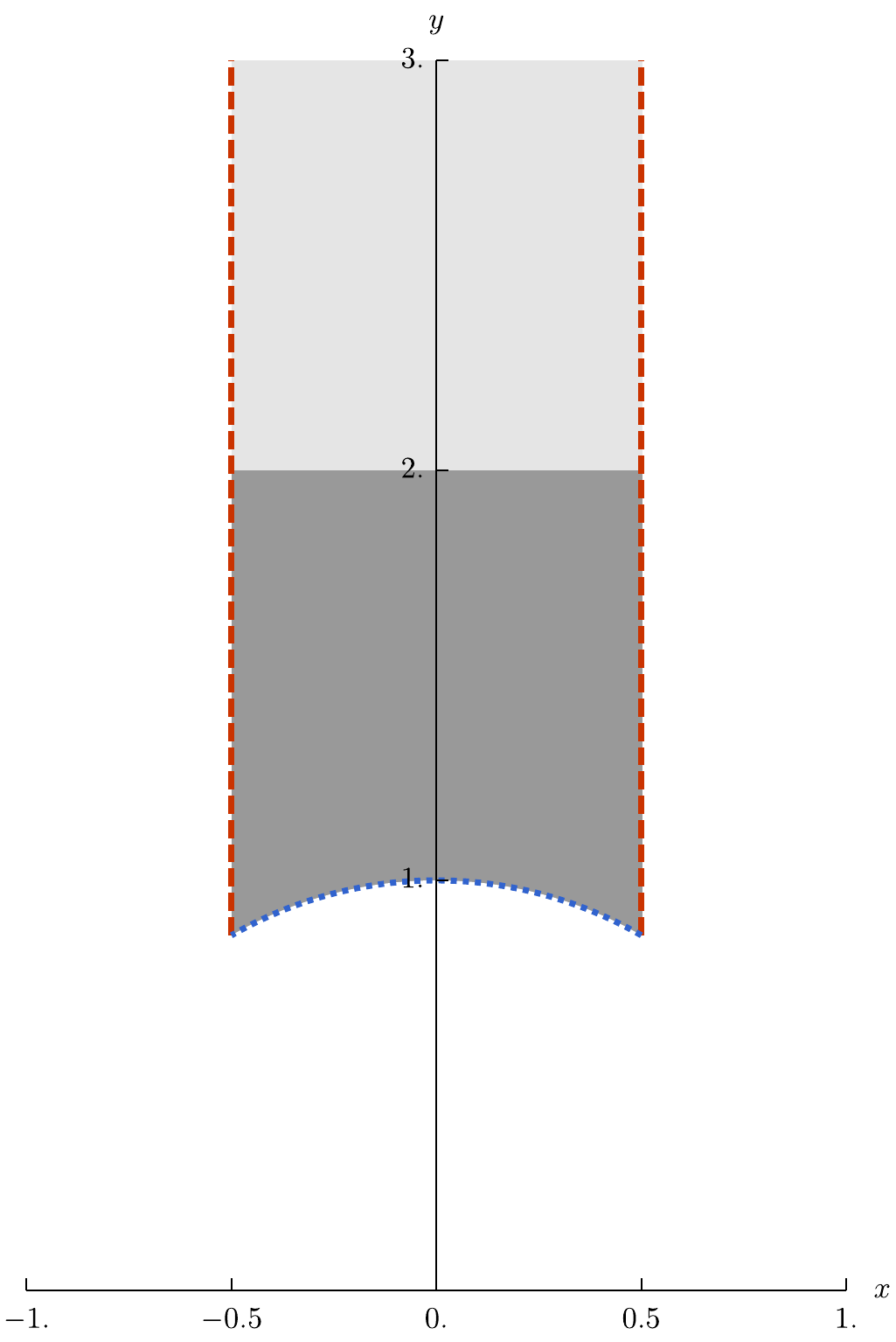}
\caption[Fundamental domain for the modular surface]{Fundamental domain for the modular surface decomposed into cusp with $a=2$ (lighter shading) and a compact part (darker shading). The arcs of the same colour/line type are identified}
 \label{fig4}
 \end{center}
\end{figure}

This surface can be decomposed into a compact part 
and a cusp as indicated in Figure \ref{fig4}.
It can also be obtained directly as a quotient $X = \mathrm{PSL}(2,\mathbb{Z}) \backslash \mathbb{H}$ as the 
above described domain is a fundamental domain of the $\mathrm{PSL}(2,\mathbb{Z})$ action, and the boundary components are
identified using the maps $z \mapsto z+1$ and $z \mapsto -\frac{1}{z}$ (see, for example,  \cite{MR1942691} for an introduction).

While the metric $y^{-2}(dx^2 + dy^2)$ has constant curvature $-1$ we can consider a function $\phi$ which is compactly supported in the interior of the shaded region
and change the metric by a conformal factor $\mathrm{e}^\phi$ to $\mathrm{e}^{\phi(x,y)} y^{-2}(dx^2 + dy^2)$. If
\[
 \int_M (1-\mathrm{e}^{\phi}) y^{-2} \dr x \dr y =0,
\]
this conformal transformation leaves the volume of $X$ unchanged. The surface equipped with this modified metric will in general have  non-constant curvature.
 
 \subsubsection{Known properties of the spectrum}

\underline{\bf Case $\phi=0$:}
In the case of constant curvature $-1$ ($\phi=0$) this surface is arithmetic. It has
infinitely many embedded eigenvalues (the so-called Maass-eigenvalues) satisfying a Weyl law as shown by Selberg \cite{selberg1989collected}
using his trace formula.
The scattering matrix $C(s)$ can be computed explicitly and equals
\begin{equation}\label{eq:Cmodular}
 C(s) = \frac{\Lambda(2s -1)}{ \Lambda(2s) },
\end{equation}
where $\Lambda(s) = \pi^{-\frac{s}{2}} \Gamma(\frac{s}{2}) \zeta(s)$ and $\zeta(s)$ is the Riemann zeta function \cite{MR0371818}, \cite{MR803366} .
Moreover, the Maass eigenvalues have been computed with great accuracy and verified by a rigorous algorithm \cite{MR2249995}, see also \cite{BSVdata}.
This surface therefore provides an excellent test for our method. 

\noindent 
\underline{\bf General case:}
In case $\phi$ is non-constant (i.e. if curvature is non-constant) one expects  at least some of the embedded eigenvalues to dissolve and become resonances
(\cite{MR1226958}).
This has become known as the Sarnak-Phillips conjecture.
Similarly the resonances will move away from the critical line.

\subsubsection{Numerical results}

\underline{\bf Case $\phi=0$:}  Since the surface is symmetric with respect to the transformation
$x \mapsto -x$ one can use symmetry reduction and consider the space of even and odd functions. These are functions
on
\[
\left \{ (x,y) \in \mathbb{H} \mid x^2+y^2 \geq 1, \; 0\leq x \leq \frac{1}{2} \right\}
\]
satisfying either Dirichlet (odd functions) or Neumann (even functions) boundary conditions at the boundary. 
The spectrum on the space of odd functions is purely discrete, and there are no resonances. Several first eigenvalues on the space of odd functions, and their comparison with the results of \cite{BSVdata} are presented in Table \ref{table:modularodd}. 

The results for the space of odd functions are below in Section \ref{subsection:Ardennumerics}.

\noindent 
\underline{\bf The curve in moduli space:} We chose the the family of the conformal factors
\[
\begin{split}
 \mathrm{e}^{\phi_q(x,y)} &= 1 + q\,c(x,y) ,\\
 c(x,y)&=\sin(5x - 0.5)\, \mathrm{e}^{-40 ((x-0.1)^2 +(y-1.5)^2)},
\end{split}
\]
with parameter $q$ in the interval $q \in [-2.,2.]$. One can sum over the group $\mathrm{PSL}(2,\mathbb{Z})$ to make this conformal factor
a function on the surface. For numerical purposes the additional terms introduced in that way are however irrelevant as they
are below working (double) precision. Note that the resulting family of metrics has constant curvature precisely at $q=0$. Moreover,
the volume is constant along this curve in the moduli space of metrics. We computed the Neumann-to-Dirichlet data at $200$ points in the parameter interval
on the cutoff surface with boundary at $a=2.2$, with $100$ discrete points on the boundary, as well as $600$ eigenvalues and their boundary data, 
and Fourier modes with $m$ between $-15$ and $15$. One can then trace the resonances as they move
along the curve, see Video \ref{video:1}.

\begin{figure}[htb!]
\begin{center} 
\includegraphics{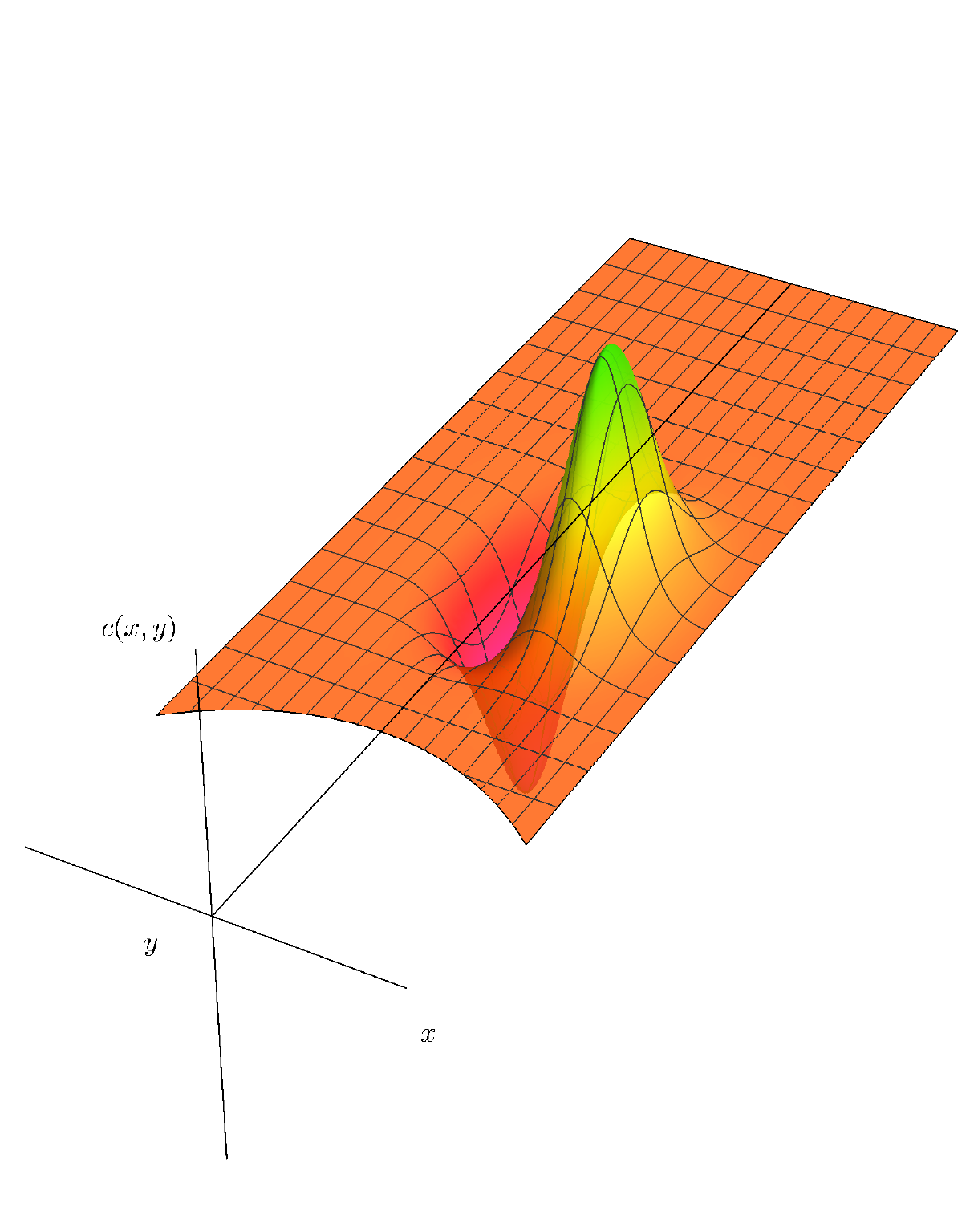}
\caption{$c(x,y)$ as a function on the modular domain}
\label{fig5}
\end{center}
\end{figure}

\begin{figure}[htb!]
\begin{center} 
\includegraphics{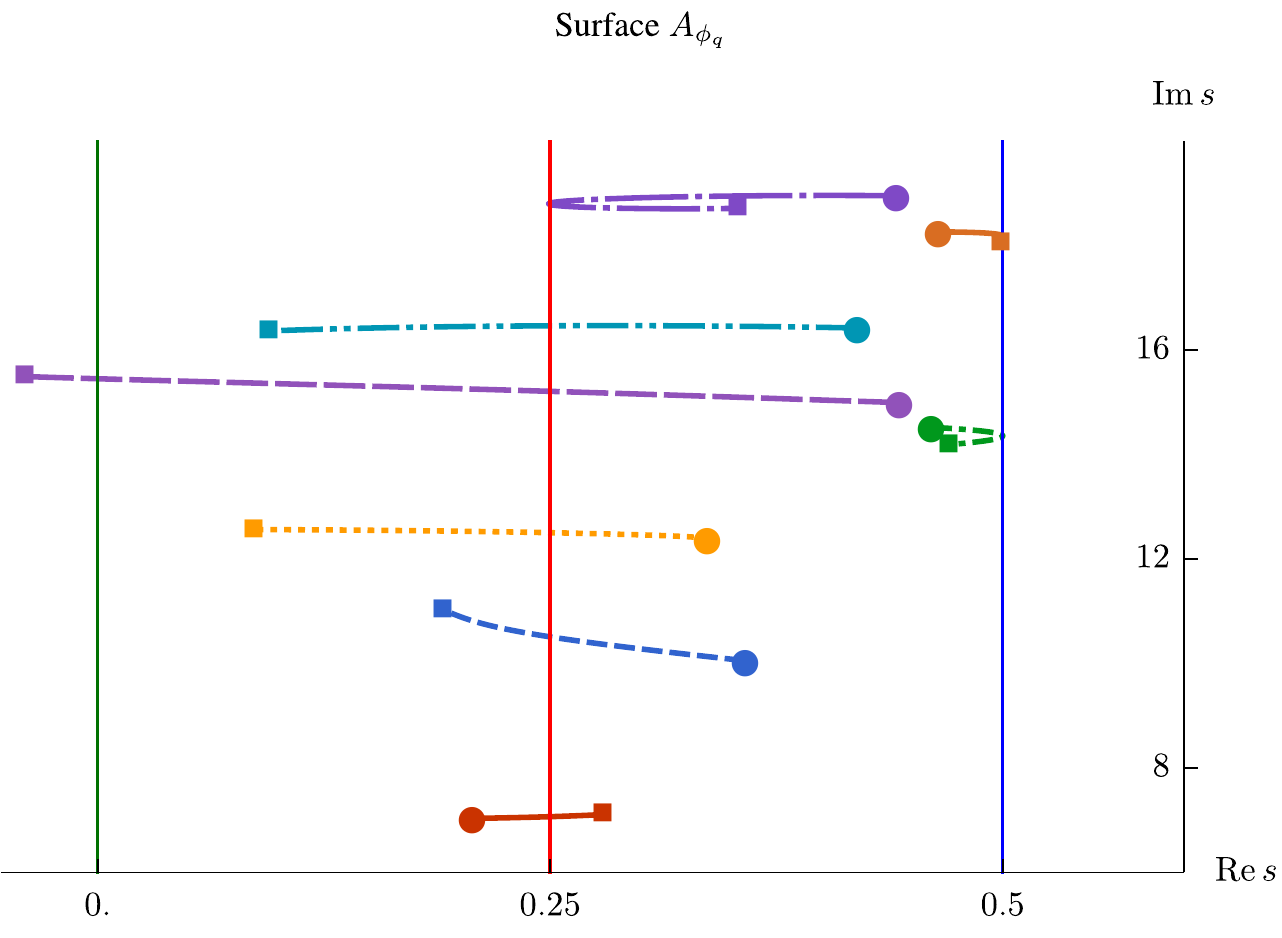}
\caption{Trajectories of eight selected resonances of $A_{\phi_q}$ as $q$ changes}
\label{fig6}
\end{center}
\end{figure}

\begin{video}[htb!] 
\centering
\ovalbox{{\href{http://michaellevitin.net/hyperbolic.html#video1}{\texttt{michaellevitin.net/hyperbolic.html\#video1}}}}

\ovalbox{{\href{https://youtu.be/pn3GvzL9ZCI}{\texttt{youtu.be/pn3GvzL9ZCI}}}}

\caption{The  dynamics of the resonances for $A_{\phi_q}$ as $q$ changes\label{video:1}}
\end{video}

The same computation was performed using the family of conformal factors
\[
\begin{split}
 \mathrm{e}^{\tilde \phi_q(x,y)}& = 1 + q\,  \tilde c(x,y) ,\\
 \tilde c(x,y)&=\sin(5 (y-1.5)) \,\mathrm{e}^{-40 ((x-0.1)^2 +(y-1.5)^2)}.
\end{split}
\]
We omit the results which are very similar.

\subsection{$B_r$. Artin's billiard}

\subsubsection{Description of the surface}

Given a positive parameter $r>\frac{1}{2}$ this surface can be obtained from the domain
\[
 \{ (x,y) \in \mathbb{H} \mid x^2+y^2 \geq r^2, \; -\frac{1}{2} \leq x \leq \frac{1}{2} \}
\]
by gluing along the boundary as follows. The sides $x=-\frac{1}{2}$ and $x=\frac{1}{2}$
are identified by means of the parallel translation $x \mapsto x+1 $. The circular arc
$\{ (x,y) \in \mathbb{H} \mid x^2+y^2 =r^2, 0 \leq x \leq \frac{1}{2} \}$ is identified  with
$\{ (x,y) \in \mathbb{H} \mid x^2+y^2 =r^2, -\frac{1}{2} \leq x \leq 0 \}$ using the map $x \mapsto -x$.
This results in hyperbolic surface with one cusp and two conical singularities. Since the surface is symmetric with respect to the transformation
$x \mapsto -x$ one can use symmetry reduction and consider the space of even and odd functions. These are functions
on
\[
\left \{ (x,y) \in \mathbb{H} \mid x^2+y^2 \geq r^2, \; 0\leq x \leq \frac{1}{2} \right\},
\]
see Figure \ref{fig7}, 
satisfying Dirichlet or Neumann boundary conditions at the boundary. Since the spectrum on the space of odd functions is pure discrete we consider here only the spectrum on the subspace of even functions.

\begin{figure}[htb!]
\begin{center} 
\includegraphics{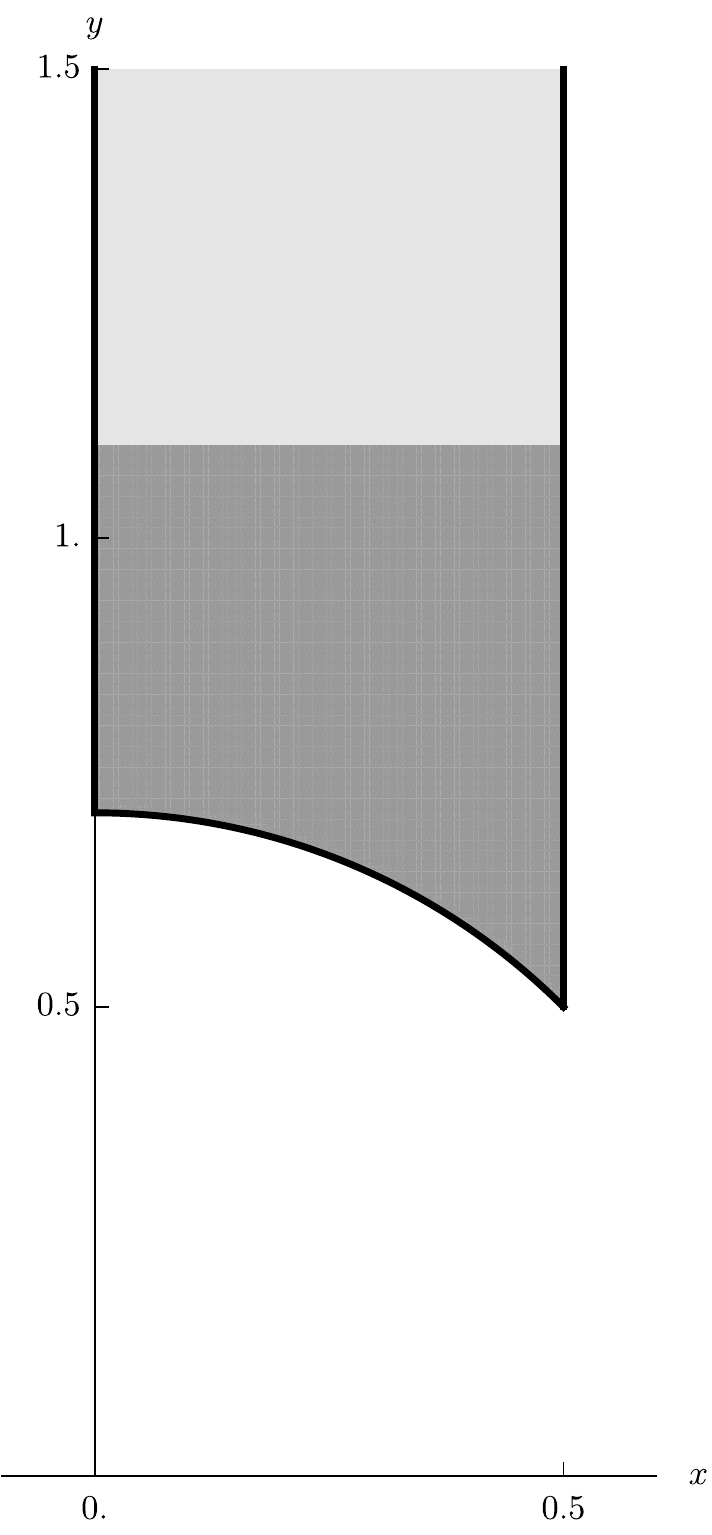}
\caption[The reduced modular domain for Artin's billiard $B_{1/\sqrt{2}}$]{The reduced modular domain for Artin's billiard $B_{1/\sqrt{2}}$. Neumann conditions are imposed on the boundary}
\label{fig7}
\end{center}
\end{figure}

\subsubsection{Known properties of the spectrum}

There are various cases when this surface can be obtained as a quotient of $\mathbb{H}$ by a Hecke triangle group $G_q$, $q\ge 3$, namely when $r^{-1}=2\cos\frac{\pi}{q}$.
In particular, for $r=\frac{1}{\sqrt{3}}$, $r = \frac{1}{\sqrt{2}}$,  and $r=1$, the resulting surfaces are arithmetic and correspond to the
surfaces obtained from the Hecke triangle groups $G_q$ in the cases $q=6,4,3$, respectively (see \cite{MR2900555}). These are  the only arithmetic cases.

They have infinitely many embedded eigenvalues satisfying Weyl's law. In each of the above three cases the scattering matrix, and hence the resonances,
can be expressed explicitly in terms of the zeros of the Riemann $\zeta$-function. This has been done explicitly in \cite{HowardThesis} but the formulae can also be deduced
using the known expressions for the congruence subgroups (\cite{hejhal1976selberg} and \cite{MR803366}). The results are
\begin{alignat}{2}
r&=\frac{1}{\sqrt{3}}:&\qquad C(s) &=  \frac{1+3^{1-s}}{1+3^{s}} \frac{\Lambda(2s -1)}{ \Lambda(2s) }\label{eq:Conj1sqrt3}\\
 r&=\frac{1}{\sqrt{2}}:& \qquad C(s) &=  \frac{1+2^{1-s}}{1+2^{s}} \frac{\Lambda(2s -1)}{ \Lambda(2s) }\label{eq:Conj1sqrt2}\\
 r&=1\ (\text{ the same as } A_0):&\qquad C(s) &=  \frac{\Lambda(2s -1)}{ \Lambda(2s) }.\label{eq:Conj1}
 \end{alignat}
  Because of the different choice of cusp-width our formulae differ by a factor $3^{1/2-s}$ and $2^{1/2-s}$ respectively from \cite{HowardThesis} in the first two cases.
 These surfaces were recently investigated in the context of the Sarnak-Phillips conjecture by Hillairet and Judge, who proved that for generic
 $r$ there are no eigenvalues (\cite{hillairet2014hyperbolic}) in the subspace of even functions.

\subsubsection{Numerical results}\label{subsection:Ardennumerics} We have computed resonances for $1000$ equidistant points in the parameter range $r \in [0.54, 1.20]$ and tracked them, see Video \ref{video:2}, and also Figure \ref{fig8} for selected resonances.

\begin{video}[htb!] 
\centering
\ovalbox{{\href{http://michaellevitin.net/hyperbolic.html#video2}{\texttt{michaellevitin.net/hyperbolic.html\#video2}}}}

\ovalbox{{\href{https://youtu.be/pn3GvzL9ZCI}{\texttt{youtu.be/pn3GvzL9ZCI}}}}

\caption{The  dynamics of the resonances for $B_r$ as $r$ changes\label{video:2}}
\end{video}
 
\begin{figure}[htb!]
\begin{center} 
\includegraphics{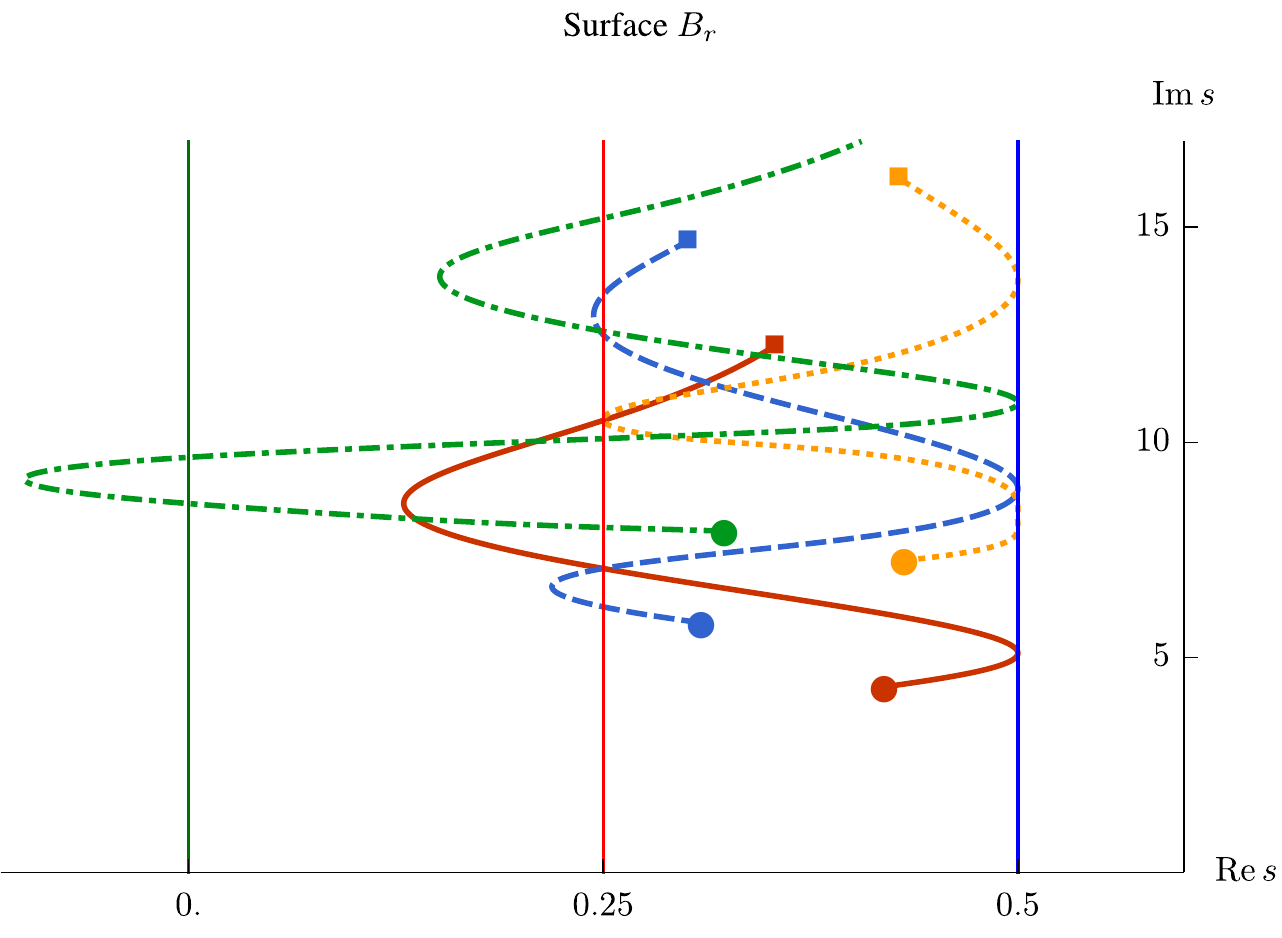}
\caption[Trajectories of four selected resonances for $B_{r}$]{Trajectories of four selected resonances for $B_{r}$, $r\in[0.54,1.20]$}
\label{fig8}
\end{center}
\end{figure}  

The resonances and embedded eigenvalues for the special arithmetic cases are shown in Figure  \ref{fig9}.

We have also investigated the case $r=0.5001$, which is close to the limiting case $r=1/2$. Since in this case another layer of continuous spectrum appears
one expects resonances to accumulate near the spectrum as $r \to \frac{1}{2}$. Apart from these resonances clustering around the spectrum we find stable
ones that seem to converge to half the Riemann zeros, see Figure \ref{fig9}. 

 \begin{figure}[htb!]
\begin{center} 
\includegraphics{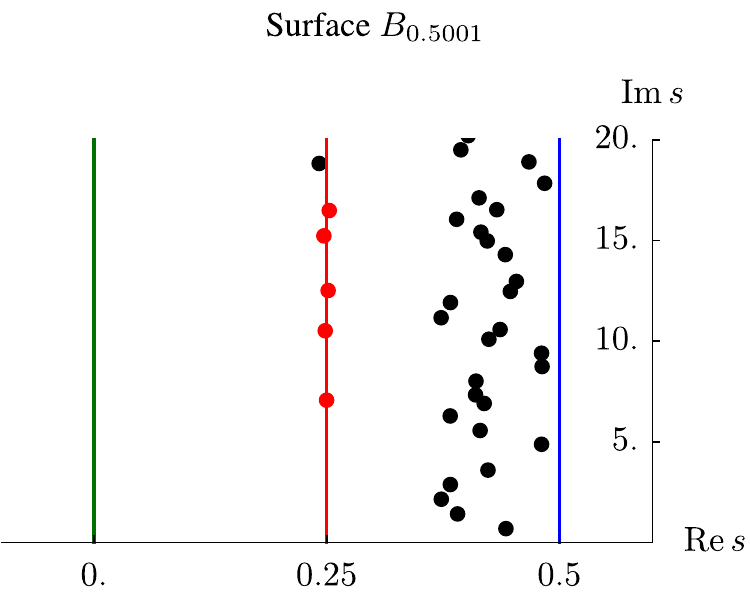}\hfill
\includegraphics{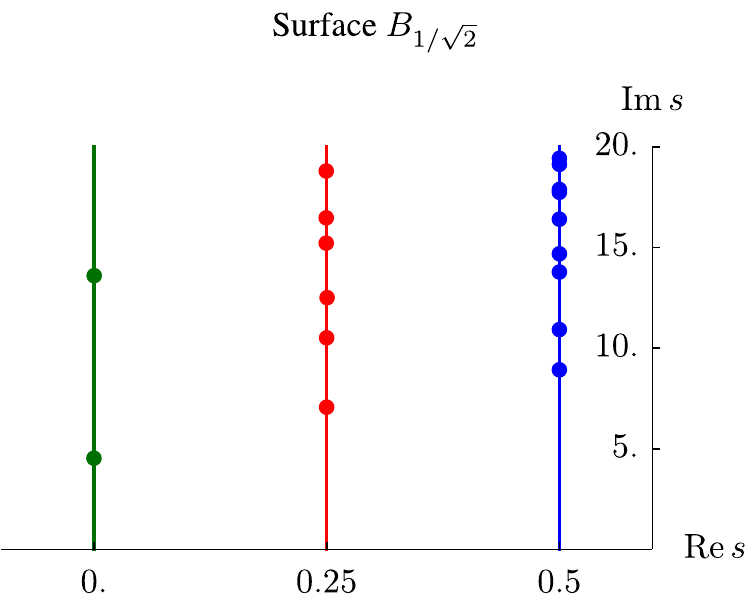}\\
\includegraphics{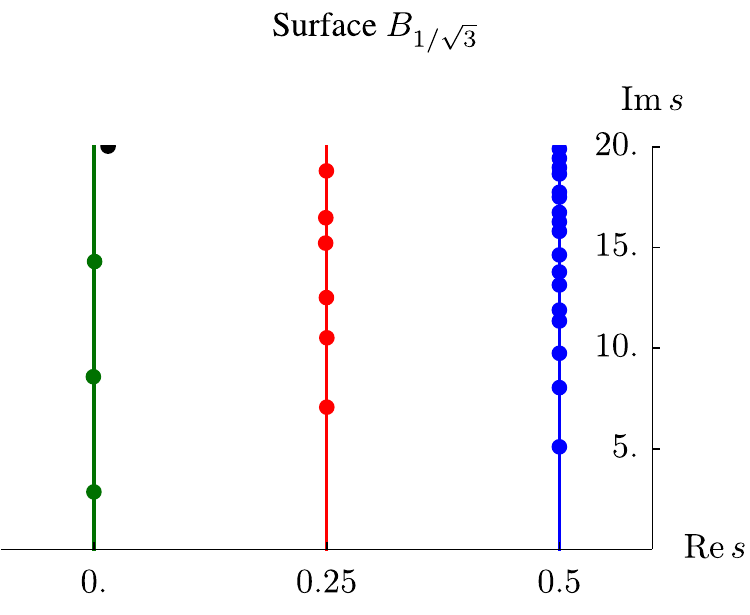}\hfill
\includegraphics{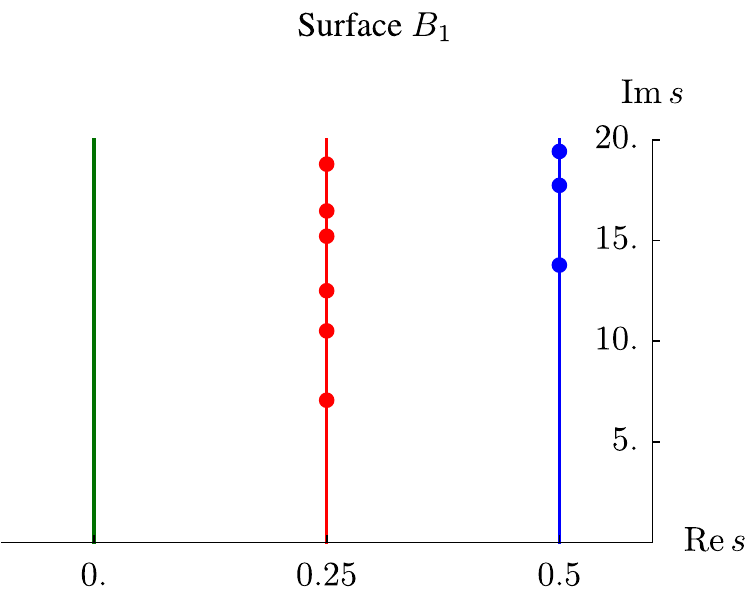}
 \caption{Resonances for $B_r$ with special values of $r$}
 \label{fig9}
 \end{center}
 \end{figure}  

The numerical values of resonances and embedded eigenvalues for four special cases, and comparison with theoretical predictions and known results are collected in Tables \ref{table:artin-even-resonances} and \ref{table:artin-ev}.

\subsection{$C_{\ell,\tau}$. Hyperbolic surfaces of genus one with one cusp}

\subsubsection{Description of the surface}

The Teichm\"uller space for genus one surfaces of constant negative curvature $-1$ and one cusp is two dimensional and can be
parameterised by the two Fenchel-Nielsen coordinates $\ell >0$ and $\tau \in [0,1)$. The parameter $\ell$ is the length of a primitive closed geodesic
and the angle $\tau$ is the twist parameter along this geodesic. Given the above two parameters we have an explicit description of the corresponding surface
of genus one with one cusp as follows. 

 For a fixed $\ell>0$, let $\alpha>0$ be the angle $\alpha=2 \arctan(\tanh{\frac{\ell}{4}})=\arcsin(\tanh{\frac{\ell}{2}})$. Then the fundamental domain
of the surface with Fenchel-Nielsen coordinates $(\ell, \tau)$ is the non-compact domain $D$ bounded by the following oriented geodesic
arcs, $\gamma_1, \gamma_2, \gamma_3, \gamma_4, \gamma_5, \gamma_6, \gamma_7$,
\[
 \begin{alignedat}{2}
  \gamma_1:  [-\alpha,\alpha]  &\to \mathbb{H}, &\quad\psi &\mapsto \frac{\sin \alpha}{4} \I \;e^{- \I\psi},\\
  \gamma_2:  \left[-\frac{\pi}{2}+\alpha,\frac{\pi}{2}-\alpha\right]  &\to \mathbb{H}, &\quad\psi &\mapsto \frac{1}{4}+\frac{\cos \alpha}{4} \I \; e^{-\I\psi},\\
  \gamma_3:\left[-\frac{\pi}{2}+\alpha,\frac{\pi}{2}-\alpha\right] &\to \mathbb{H}, &\quad\psi &\mapsto -\frac{1}{4}+\frac{\cos \alpha}{4} \I \; e^{\I\psi},\\
  \gamma_4:  [-\alpha,0]  &\to \mathbb{H},&\quad\psi &\mapsto \frac{1}{2}+\frac{\sin \alpha}{4} \I \; e^{-\I\psi},\\
   \gamma_5:  [0,\alpha]  &\to \mathbb{H},&\quad\psi &\mapsto -\frac{1}{2}+\frac{\sin \alpha}{4} \I \; e^{\I\psi},\\
   \gamma_6: \left[\frac{\sin \alpha}{4},\infty\right) &\to \mathbb{H},&\quad\psi &\mapsto \frac{1}{2} + \I\psi,\\
   \gamma_7: \left[\frac{\sin \alpha}{4}, \infty \right) &\to \mathbb{H}, &\quad\psi &\mapsto +\frac{1}{2} + \I\psi;
 \end{alignedat} 
 \]
for brevity, we use the complex coordinate $x+ \I y$ on $\mathbb{H}$.

 Note that $ \gamma_2, \gamma_4$ and $\gamma_6$ are the images of $\gamma_3, \gamma_5$ and $\gamma_7$ respectively
 under the reflection about the $y$-axis $x+ \I y \mapsto -x + \I y$. 
 Figure \ref{fig10} depicts the fundamental domain decomposed into a compact part (darker shading)  and a cusp (lighter shading).
 The surface $C_{\ell,\tau}$ is formed as follows.
 The infinite geodesic  $\gamma_6$ is identified with $\gamma_7$ using the hyperbolic translation $z \mapsto z - 1$.
 The geodesic arc $\gamma_2$ is identified with $\gamma_3$ using the hyperbolic motion along $\gamma_1$. 
 Once these identifications are completed both $\gamma_1$ and $\gamma_4 \cup \gamma_5$ become closed boundary geodesics
 of length $\ell$.  These boundary components can be glued together as follows. First shift all points on $\gamma_1$ by $\tau \ell$.
 Then use hyperbolic translation along $\gamma_2$ and $\gamma_3$ to map the geodesic onto $\gamma_4 \cup \gamma_5$.
 The resulting surface $C_{\ell,\tau}$ is a surface of genus one with one cusp such as the one depicted in Figure \ref{fig2}.  The light-shaded region in Figure  \ref{fig10} gives a hyperbolic surface $M$ of genus one with horocyclic boundary.
 
\begin{figure}[htb!]
\begin{center} \includegraphics{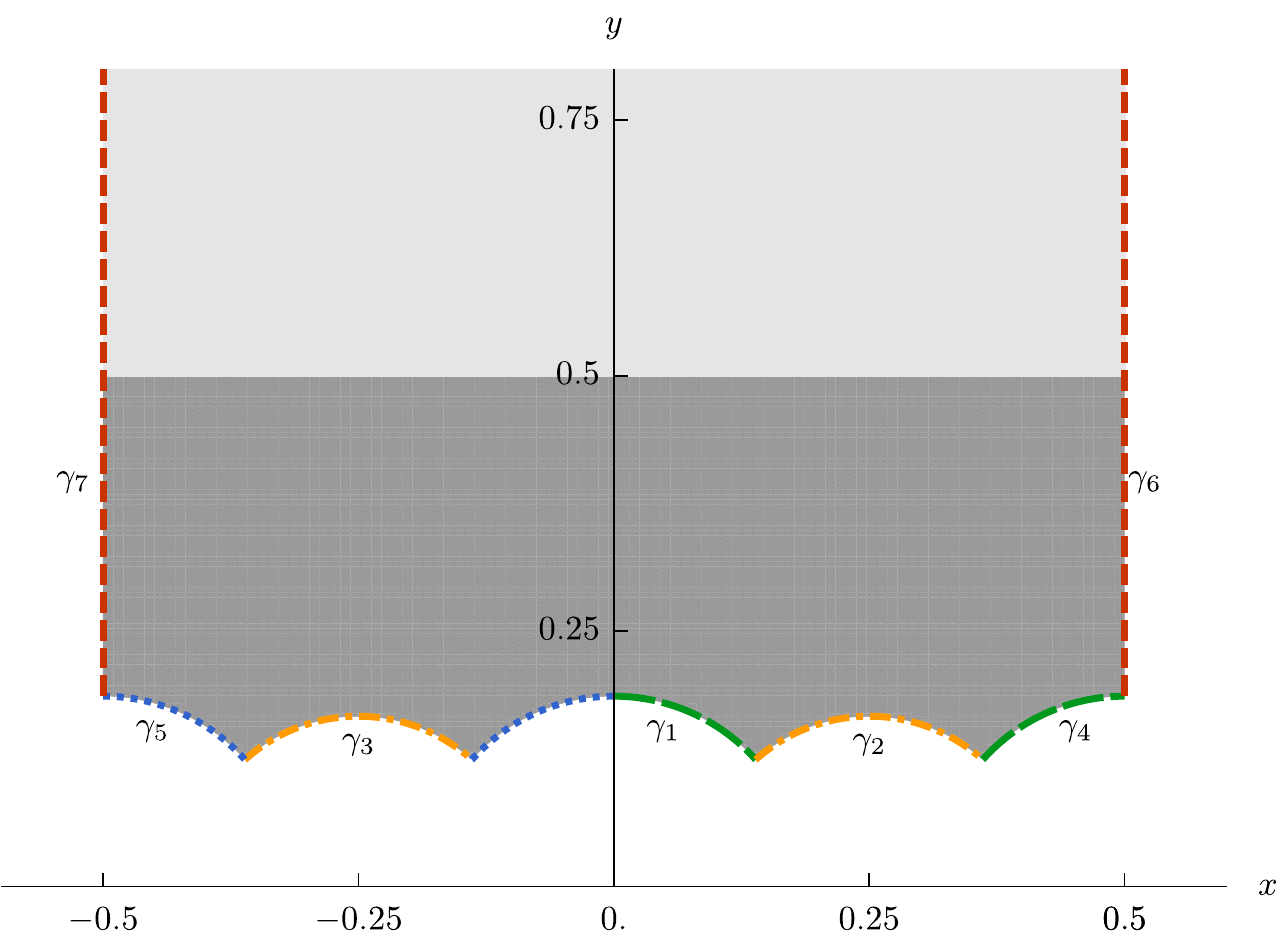}
\caption[Fundamental domain for a hyperbolic surface of genus one with one cusp]{Fundamental domain for a hyperbolic surface of genus one with one cusp, shown here is $C_{2\arccosh\left(\frac{3}{2}\right), \frac{1}{2}}$. Identified boundary arcs are shown in the same colour. The shading indicates the decomposition into a cusp and a compact part}
\label{fig10}
\end{center}
\end{figure}
 
In addition to the closed geodesics of length $\ell$, $C_{\ell,\tau}$ has another closed geodesics of length
\begin{equation}\label{eq:ell'}
\ell'=\ell'(\ell,\tau)=\arccosh\left(\frac{\cosh(\ell\tau)\left(\cosh\left(\frac{\ell}{2}\right)\right)^2+1}{\left(\sinh\left(\frac{\ell}{2}\right)\right)^2}\right).
\end{equation}   
The two lengths are equal whenever 
\[
\tau=\tau^*(\ell)=\frac{\arccosh(\cosh\ell -2)}{\ell},
\]
or equivalently when $\ell=\ell^*(\tau)$ is the positive solution of
\[
\cosh\ell=2+\cosh(\ell\tau).
\]  

 \subsubsection{Known properties of the spectrum}
 
 In general the Laplace operator on the surface will have simple continuous spectrum and may have embedded eigenvalues. The expectation is however, that these embedded are generically absent. There are several special cases for which the surface $C_{\ell,\tau}$ is symmetric, and therefore a symmetry reduction can be performed. We will single out and discuss several particular families.
 
 \subsubsection{Numerical results,  case 1: $\tau=0$, varying $\ell$}

With the twist parameter $\tau$ fixed the only remaining parameter is the length parameter $\ell$. Since in this case the twist is zero, the curve $\gamma_2$ becomes a closed simple geodesic of length 
 \begin{equation}\label{eq:ellprime0}
\ell'=\ell'(\ell,0)=\arccosh\left(1+\frac{2}{\sinh^2\left(\frac{\ell}{2}\right)}\right).
 \end{equation}
on the resulting surface. As the function $\ell'(\ell,0)$ is monotone decreasing in $\ell$, $\ell'(\ell'(\ell,0),0)\equiv \ell$, and also $\ell=\ell'(\ell,0)$ when $\ell=\ell^*(0)=\arccosh(3) \approx 1.762747$, our parametrisation of $C_{\ell,0}$ is not unique: namely, the surfaces $C_{\ell,0}$ and $C_{\ell'(\ell,0),0}$ are always isometric. Therefore it only makes sense to track resonances for $\ell\le \ell^*(0)$. We have nevertheless analysed some special values of $\ell> \ell^*(0)$ to verify that our numerical results do not depend on the choice of parametrisation.

We have tracked the resonances in the interval $\ell \in [1.2,\arccosh(3)]$, see Video \ref{video:3}, and also  Figure \ref{fig11} for the trajectories traced by four selected resonances. 

\begin{video}[htb!] 
\centering
\ovalbox{{\href{http://michaellevitin.net/hyperbolic.html#video3}{\texttt{michaellevitin.net/hyperbolic.html\#video3}}}}

\ovalbox{{\href{https://youtu.be/Li6Azx01IG4}{\texttt{youtu.be/Li6Azx01IG4}}}}

\caption{The  dynamics of the resonances for $C_{\ell,0}$ as $\ell$ changes\label{video:3}}
\end{video} 
 
 \begin{figure}[htb!]
\begin{center} 
\includegraphics{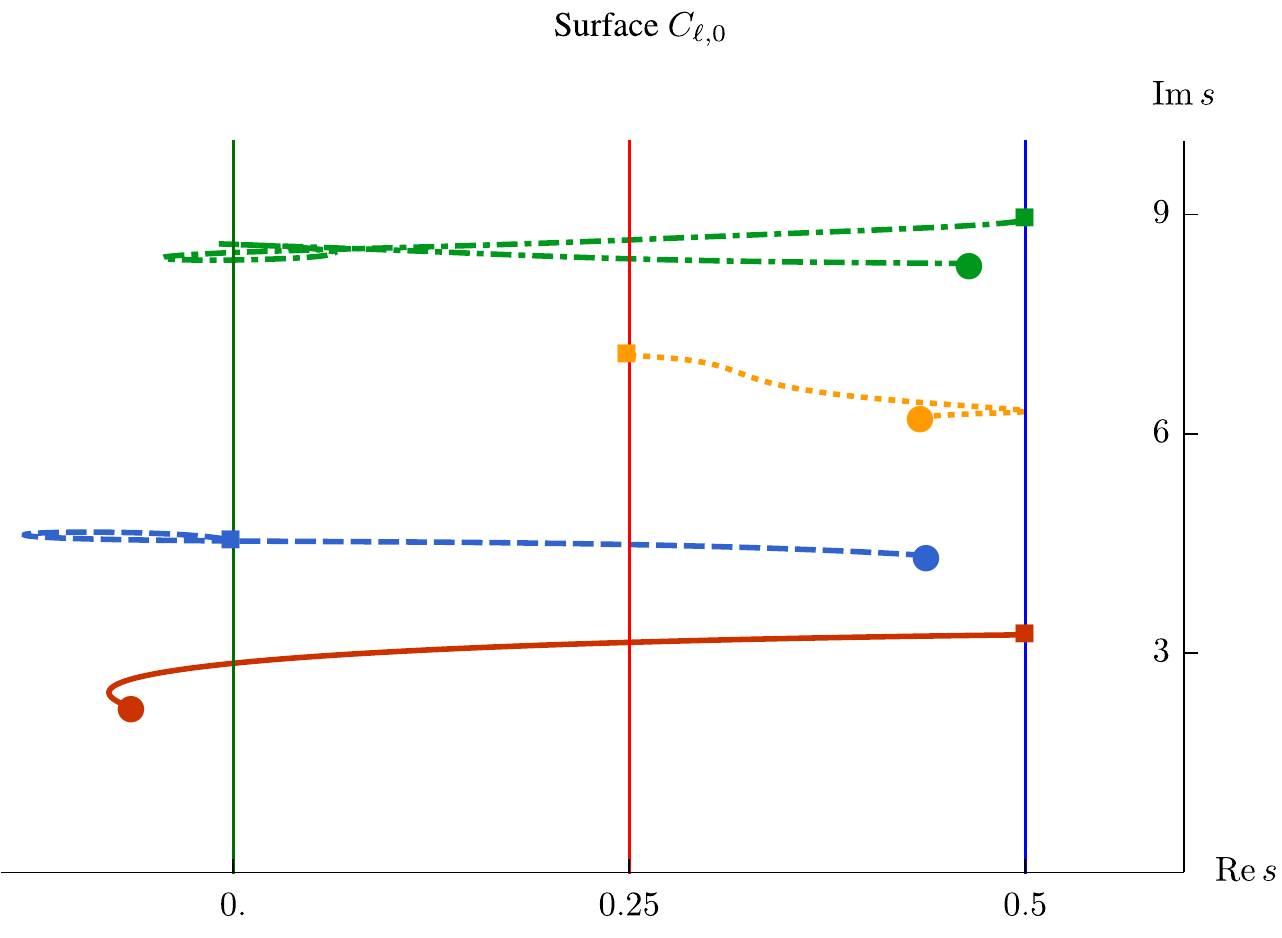}
 \caption[Trajectories of four selected resonances for $C_{\ell, 0}$]{Trajectories of four selected resonances for $C_{\ell, 0}$, $\ell\in[1.2,\arccosh(3)]$}
 \label{fig11}
 \end{center}
 \end{figure}  

In this interval there are several special lengths. The numerically found resonances and the first ten embedded eigenvalues for these special lengths are in Tables \ref{table:genus-one-no-twist-special-ell} and  \ref{table:genus-one-special-ell-evalues}, resp.

\subsubsection*{\underline{$\ell=\arccosh(2) \approx 1.316957$}}
In this special case the scattering matrix takes the form
 \begin{equation}\label{eq:no-twist-C-special1}
  C(s) = 2^{1-2s} \frac{1+2^{1-s}}{1+2^{s}}  \frac{1+3^{1-s}}{1+3^{s}} \frac{\Lambda(2s -1)}{ \Lambda(2s) }.
\end{equation}
Indeed, one can conjugate the generators of the corresponding Fuchsian group into the following three matrices
\[
\frac{1}{\sqrt{6}} \left( \begin{matrix} 3 & 3 \\ 6 & 3 \end{matrix} \right), \; \frac{1}{\sqrt{6}} \left( \begin{matrix} 6 & 1 \\ -6 & 0 \end{matrix} \right),\; \left( \begin{matrix} 1 & 2 \\ 0 & 1 \end{matrix} \right).
\]
These clearly generate a subgroup $\Gamma$ of the arithmetic group $\tilde{\Gamma}_0(6)$, that is the group  generated by $\Gamma_0(6)$ together with its Atkin-Lehner (Fricke) involutions.
Since the Atkin-Lehner involutions act transitively on the four cusps of $\Gamma_0(6) \backslash \mathbb{H}$ the domain $\tilde \Gamma_0(6) \backslash \mathbb{H}$
has only one cusp. The scattering matrix for $\Gamma_0(N)$ has been computed in \cite{hejhal1976selberg} and \cite{MR803366}. If $N$ is square-free then according to Hejhal  \cite[Vol 2, p 536]{hejhal1976selberg}
the full scattering matrix equals
\begin{gather} \label{scatterrefprev} 
 C(s) = \frac{\Lambda(2s -1)}{ \Lambda(2s)} \bigotimes_{\substack{q \vert N\\q\text{ prime}}} M_q(s),
\end{gather}
where
\[
M_q(s) = \frac{1}{q^{2s}-1} \left( \begin{matrix} q-1 & q^s-q^{1-s} \\ q^s-q^{1-s}  &  q-1  \end{matrix}\right).
\]
The vector  $\left( \begin{matrix}  1 \\ 1 \end{matrix} \right)$ is an eigenvector of $M_q(s)$ with eigenvalue $\frac{1+q^{1-s}}{1+q^s}$. 
The Atkin-Lehner involutions act transitively on the cusps for square-free $N$. Therefore, in this case, $\tilde \Gamma_0(N)$, the group generated by the Atkin-Lehner involutions and $\Gamma_0(N)$, gives rise to a quotient with one cusp. The scattering matrix for $\tilde \Gamma_0(N)$  must then be the restriction of the scattering matrix to functions invariant under the Atkin-Lehner involutions. The invariant functions correspond to the the span of the vector $\bigotimes\limits_{q \mid N}\left( \begin{matrix}  1 \\ 1 \end{matrix} \right)$ in this representation of the scattering matrix. This vector is an eigenvector of $\bigotimes\limits_{q \vert N} M_q(s)$ with eigenvalue  $\prod\limits_{q \vert N} \frac{1+q^{1-s}}{1+q^s}$.
Summarising, in the case of square-free $N$, the scattering matrix for $\tilde \Gamma_0(N)$ is given by
\begin{equation}\label{scatterref} 
 C(s) =  \frac{\Lambda(2s -1)}{ \Lambda(2s)} \prod_{\substack{q \vert N\\q\text{ prime}}} \frac{1+q^{1-s}}{1+q^{s}}.
\end{equation}
This has also been obtained in \cite[Lemma 5]{JST13}.
Since  $\tilde \Gamma_0(6)$ acts on the fundamental domain for our group but leaves the cusp invariant,
the scattering matrix for $\Gamma$ must be the same, apart from the extra factor $2^{1-2s}$ appearing because of the cusp width $2$.
We refer to \cite{venkov} for details of this argument.

Figure \ref{fig13} shows the computed resonances for $\ell=\arccosh(2)$.

\subsubsection*{\underline{$\ell=\arccosh(3) \approx 1.762747$}}
In this case $\alpha=\dfrac{\pi}{4}$, so all the boundary arcs of $C_{\ell, 0}$ have the same radius $\dfrac{\sqrt{2}}{8}$. Also, the length $\ell'(\ell,0)$ of the second distinguished closed geodesic $\gamma_2$ coincides with $\ell$. We can carry out the following sequence of symmetry reductions, see Figure \ref{fig12}.

\begin{figure}[htb!]
\begin{center} 
\includegraphics{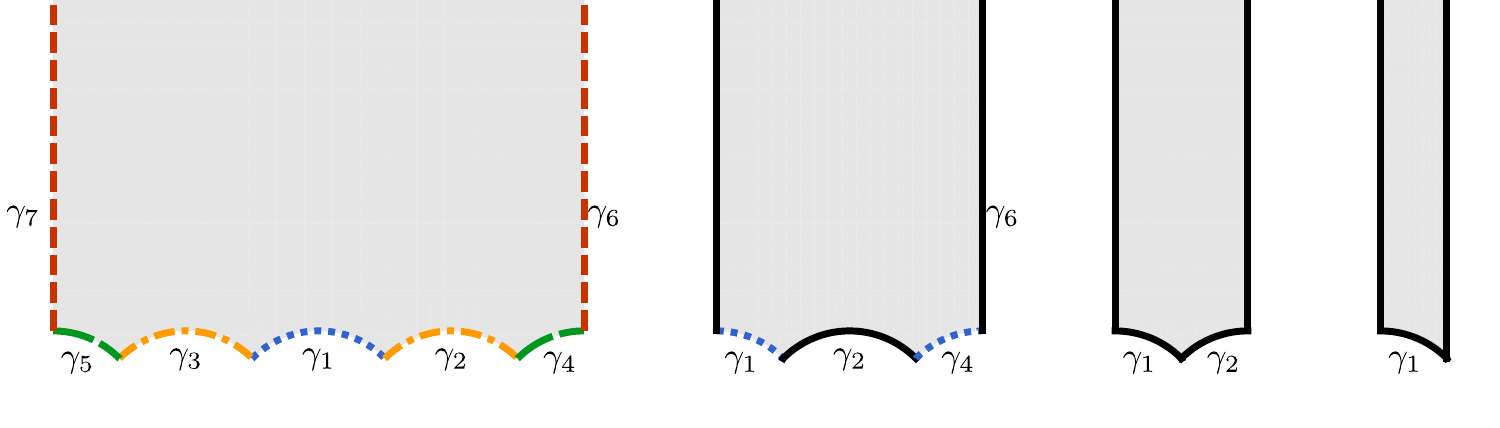}
\caption[Sequence of symmetry reductions of $C_{\ell, 0}$, $\ell=\arccosh(3)$]{Sequence of symmetry reductions of $C_{\ell, 0}$, $\ell=\arccosh(3)$. The arcs  are identified with the dashed counterpart of the same colour. Neumann condition is imposed on arcs coloured black}
\label{fig12}
\end{center}
 \end{figure}

 First of all the domain $X=C_{\ell, 0}$ has a reflection symmetry $x \mapsto -x$. It follows that we have a natural decomposition of $L^2(X)$ into invariant subspaces for $\Delta$
 consisting of even and odd functions. On the subspace of odd functions the spectrum of $\Delta$ is discrete since it is part of the space of cups forms, and the Eisenstein series are all even.
 The subspace of even functions corresponds to the space of functions on half of the domain $\{ z = x+ \I y \in X \mid x \geq 0 \}$ satisfying Neumann boundary conditions
 along $\gamma_6$, $\gamma_2$ and along the $y$-axis part of the boundary of the resulting domain and periodic boundary conditions
 that identify $\gamma_4$ with the right half of $\gamma_1$. Now we have another symmetry
 $x \mapsto \dfrac{1}{4}-x$. Again the space of odd functions is contained in the space of cusp forms and the even part corresponds to considering the domain
 $\left\{ z = x+ \I y \in X \mid 0\le x \le \dfrac{1}{4}\right\}$ with Neumann boundary conditions everywhere along its boundary. The resulting domain has yet another symmetry
 $x \mapsto \dfrac{1}{4}-x$. The Laplace operator on the space of even functions on this domain corresponds the Laplace operator on the domain
 $\left\{ z = x+ \I y \in D \mid 0\le x \leq \dfrac{1}{8}\right \}$ with Neumann boundary conditions everywhere along its boundary.  Since the symmetry reduction of Artin's billiard
 for $r = \dfrac{1}{\sqrt{2}}$ leads, after scaling by a factor $\dfrac{1}{4}$, to an isometric domain this shows that the continuous spectral subspace of $C_{\arccosh(3),0}$ and that of $B_{\frac{1}{\sqrt{2}}}$ are unitarily equivalent  and the scattering matrices as well as the resonances coincide up to a scaling factor. One therefore has
 \begin{equation}\label{eq:no-twist-C-special2}
  C(s) =  4^{1-2s} \frac{1+2^{1-s}}{1+2^{s}} \frac{\Lambda(2s -1)}{ \Lambda(2s) }.
\end{equation}
 The above discussion also shows that the discrete spectrum
 consists of several parts, each belonging to mixed Dirichlet-Neumann problems on certain domains.
 
Figure \ref{fig13} shows the computed resonances for $\ell=\arccosh(3)$.

In a similar way as before one can conjugate the generators of the corresponding Fuchsian group into 
\[
 \frac{1}{\sqrt{2}} \left( \begin{matrix} 2 & 1 \\ 2 & 2 \end{matrix} \right), \; \frac{1}{\sqrt{2}} \left( \begin{matrix} 4 & 1 \\ -2 & 0 \end{matrix} \right),\; \left( \begin{matrix} 1 & 4 \\ 0 & 1 \end{matrix} \right),
\]
which is a subgroup of the arithmetic group $\tilde \Gamma_0(2)$. Equation \eqref{eq:no-twist-C-special2} can therefore also be derived from \eqref{scatterref} in the same way as before.

\subsubsection*{\underline{$\ell=\arccosh(5) \approx 2.29243$}}
This case is isometric to the case $\ell=\arccosh(2)$ since these two lengths are related by \eqref{eq:ellprime0}, see also the discussion following that formula.

\subsubsection*{\underline{$\ell=\arccosh(9) \approx  2.88727$}}
This case is isometric to the case $\ell=\left(\arccosh\left(\frac{3}{2}\right)\right)\approx0.962424$ which lies outside our computed range. In this case the scattering matrix is given by
 \begin{equation}\label{eq:no-twist-C-special3}
  C(s) = 2^{1-2s}  \frac{1+5^{1-s}}{1+5^{s}} \frac{\Lambda(2s -1)}{ \Lambda(2s) }.
\end{equation}
The generators of the corresponding Fuchsian group can be conjugated to
\[
 \frac{1}{\sqrt{5}} \left( \begin{matrix} 15 & 8 \\ -10 & -5 \end{matrix} \right), \; \frac{1}{\sqrt{5}} \left( \begin{matrix} 0 & -1 \\ 5 & 5 \end{matrix} \right),\; \left( \begin{matrix} 1 & 2 \\ 0 & 1 \end{matrix} \right).
\]
These therefore generate a subgroup $\Gamma$ of $\tilde \Gamma_0(5)$. The surface $\tilde \Gamma_0(5) \backslash \mathbb{H}$ has one cusp. The group $\tilde \Gamma_0(5)$ acts on our surface $G \backslash \mathbb{H}$ 
and the action fixes the cusp. This implies that the scattering matrices of $G$ and of $\tilde \Gamma_0(5)$ coincide
modulo a possible factor coming from the normalisation of the cusp-width. 
In the same way as before, equation  \eqref{scatterref} (see also \cite[equation (5)]{MR2398794}) gives the formula \eqref{eq:no-twist-C-special3}. 

Figure \ref{fig13} shows the computed resonances for $\ell=\arccosh(9)$.
 
\begin{figure}[htb!]
\begin{center} 
\includegraphics{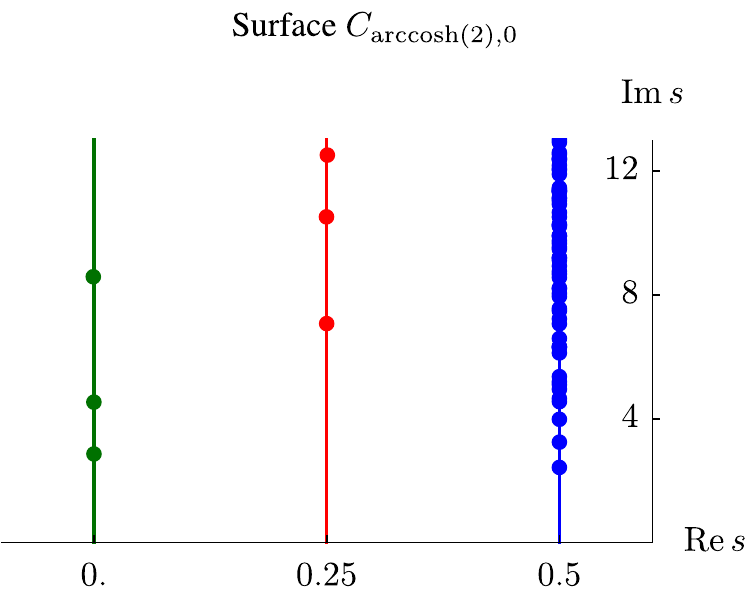}\hfill\includegraphics{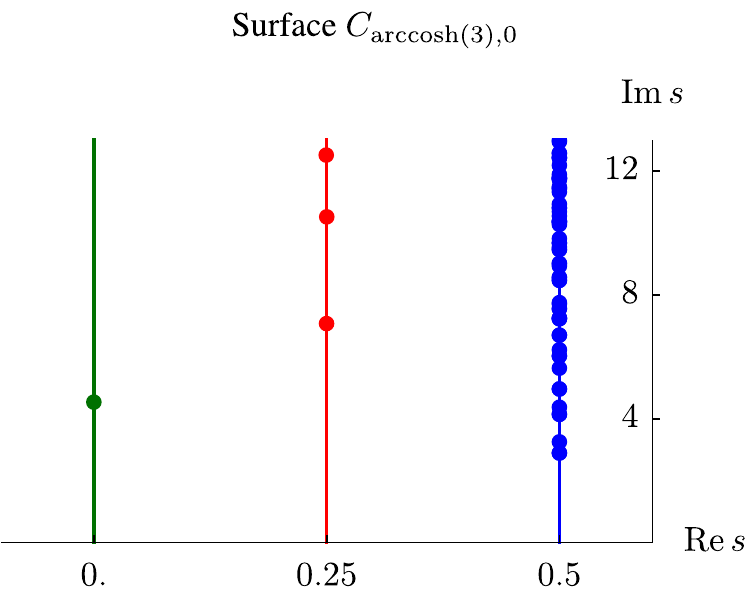}\\
\includegraphics{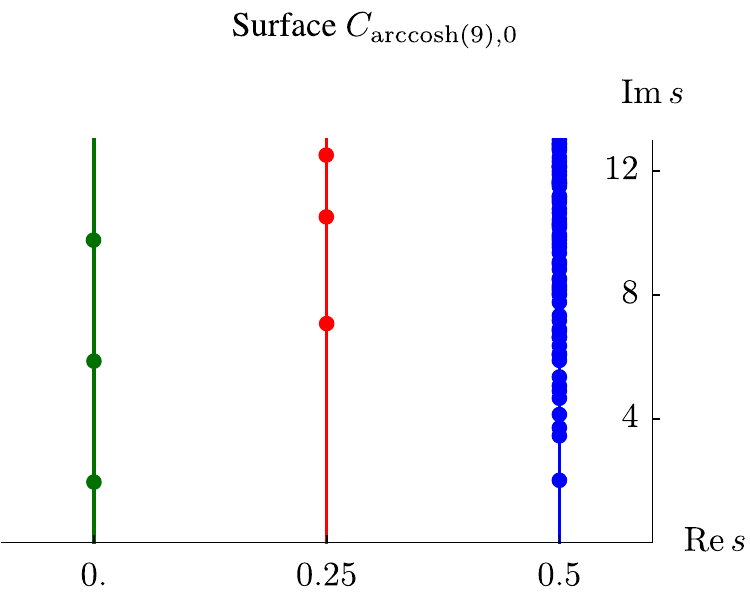}
 \caption{Resonances and embedded eigenvalues for $C_{\ell, 0}$ with special lengths $\ell$}
 \label{fig13}
 \end{center}
 \end{figure} 
 
 \subsubsection{Numerical results,  case 2: $\tau=\frac{1}{2}$, varying $\ell$}
 
 We have tracked the resonances in the interval $\ell \in [1.12485,2.72485]$, see Video \ref{video:4}, and also  Figure \ref{fig14} for the trajectories traced by four selected resonances. 

\begin{video}[htb!] 
\centering
\ovalbox{{\href{http://michaellevitin.net/hyperbolic.html#video4}{\texttt{michaellevitin.net/hyperbolic.html\#video4}}}}

\ovalbox{{\href{https://youtu.be/2kn2ZWYObAE}{\texttt{youtu.be/2kn2ZWYObAE}}}}

\caption{The  dynamics of the resonances for $C_{\ell,1/2}$ as $\ell$ changes\label{video:4}}
\end{video}

\begin{figure}[htb!]
\begin{center} 
\includegraphics{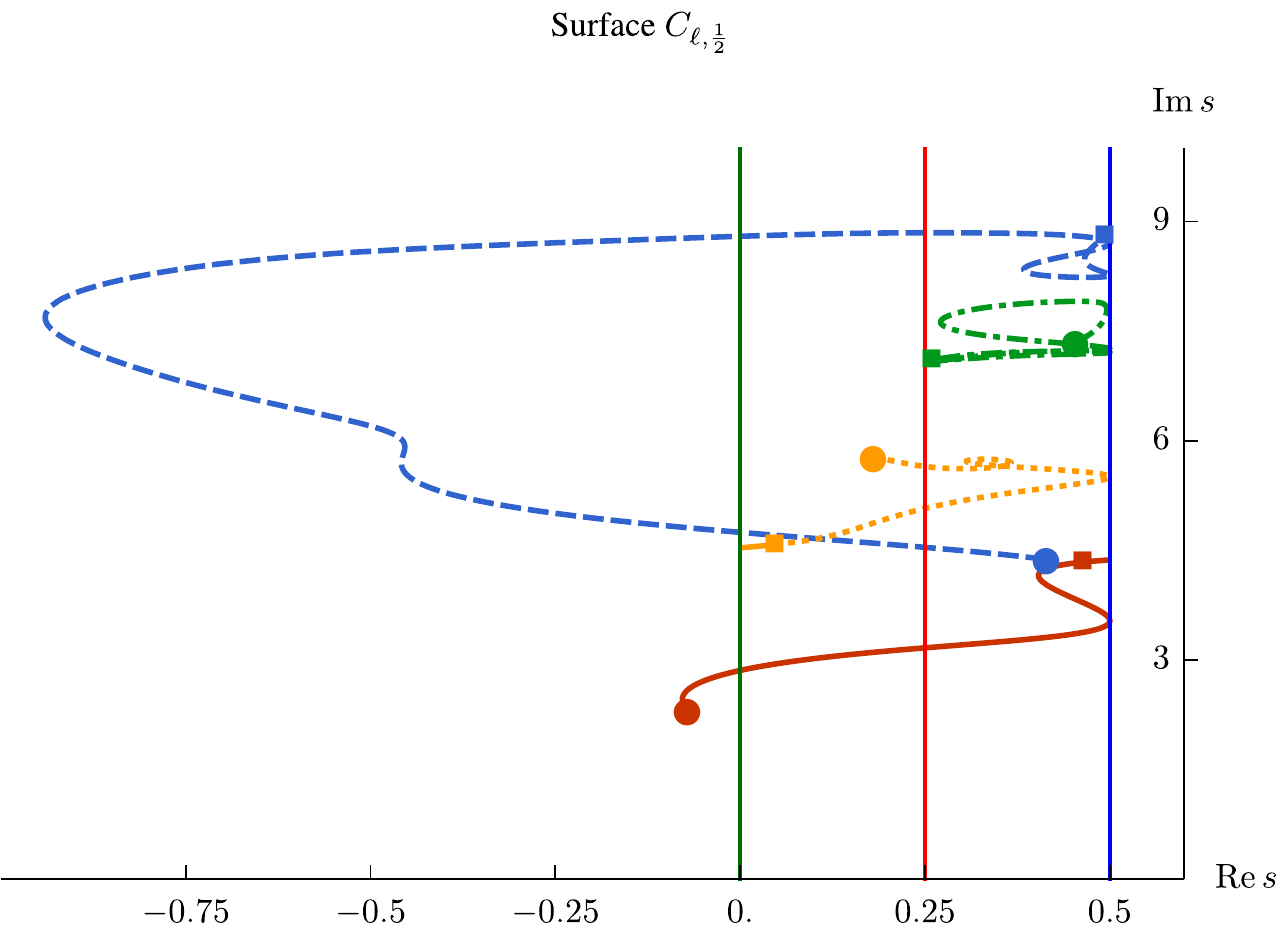}
\caption[Trajectories of four selected resonances for $C_{\ell, 1/2}$]{Trajectories of four selected resonances for $C_{\ell, 1/2}$, $\ell\in[1.12485,2.72485]$}
\end{center}
\label{fig14}
\end{figure}  
 
For twist $\tau=\frac{1}{2}$ we find the following special lengths.  The numerically found resonances and the first ten embedded eigenvalues for these special lengths are in Tables \ref{table:genus-one-half-twist-special-ell} and  \ref{table:genus-one-special-ell-evalues}, resp.

\subsubsection*{\underline{$\ell=2\arccosh\left(\frac{3}{2}\right) \approx 1.924847$}}
One can check by direct computation that for this particular $\ell$ the twist parameter $\tau=\frac{1}{2}$ is the unique twist for which the length of the second simple closed geodesic
 generating the fundamental group coincides with $\ell$. Any hyperbolic surface of genus one with one cusp that possesses two simple closed curves of that length 
  that intersect in one point only will therefore be isometric to this surface. In particular, it is isometric to the arithmetic one punctured torus described by Cohn in \cite{greenberg1974discontinuous} and by Gutzwiller in \cite{gutzwiller1983stochastic}. The scattering matrix is known to be equal to
\begin{equation}\label{eq:half-twist-C-special1}
   C(s) = 6^{1-2s}  \frac{\Lambda(2s -1)}{ \Lambda(2s) },
\end{equation}
  where the extra factor $6^{1-2s}$ relative to \cite{gutzwiller1983stochastic} is because the cusp width in \cite{gutzwiller1983stochastic} was chosen to be $6$ rather than one.
  Its scattering resonances coincide with the one for the modular domain and are there directly related to the non-trivial zeros of the Riemann zeta function. Figure \ref{fig15} shows the computed resonances for $\ell=2\arccosh\left(\frac{3}{2}\right)$.
 The form of the scattering matrix \eqref{eq:half-twist-C-special1} can also be derived as follows. 
The generators of the Fuchsian group can be conjugated into
\[
\left( \begin{matrix} 2 & 1 \\ 1 & 1 \end{matrix} \right), \;  \left( \begin{matrix} 0 & -1 \\ 1 & 3 \end{matrix} \right),\; \left( \begin{matrix} 1 & 6 \\ 0 & 1 \end{matrix} \right),
\]
which is a subgroup of $\mathrm{PSL}(2,\R)$. Therefore, $\mathrm{PSL}(2,\R)$ acts on our surface and fixes the cusp. Hence, the scattering matrix coincides with that of the modular domain up to a factor $6^{1-2s}$, since the generator
 $ \left( \begin{matrix} 1 & 6 \\ 0 & 1 \end{matrix} \right)$ yields a cusp of width $6$.

The special values close to the critical line and to the imaginary line are compared with theoretical prediction of \eqref{eq:half-twist-C-special1} in Table \ref{table:genus-one-half-twist-special-ell}.  
 
 We list some embedded eigenvalues for the twist parameter $\tau=\frac{1}{2}$ in Table \ref{table:genus-one-special-ell-evalues}. Note that some of the double eigenvalues coincide with those for the group $\Gamma^3$ from \cite{str}. Additionally, some embedded eigenvalues for the twist parameter $\tau=\frac{1}{2}$
  have been computed in \cite{kar2013computation}, however the authors have missed quite a few embedded eigenvalues in their list. They correctly identify
  two multiplicity two eigenvalues at 2.95648 and 4.51375, but do for example miss the multiplicity two eigenvalue at about 3.53606 and the simple
  eigenvalue at about 3.70339, cf. Table \ref{table:genus-one-special-ell-evalues}.  We have performed a heuristic check using Weyl's law and Turing's method
  and our list appears to be complete. 

\subsubsection*{\underline{$\ell=2\arccosh(2) \approx 2.6339157$}}

This case can be shown to be isometric to the surface $C_{\arccosh{3},0}$ by computing the generators, and our independent numerical results are in full agreement.

  Figure \ref{fig15} shows the computed resonances for  $\ell=2\arccosh(2)$.
 
\begin{figure}[htb!]
\begin{center} 
\includegraphics{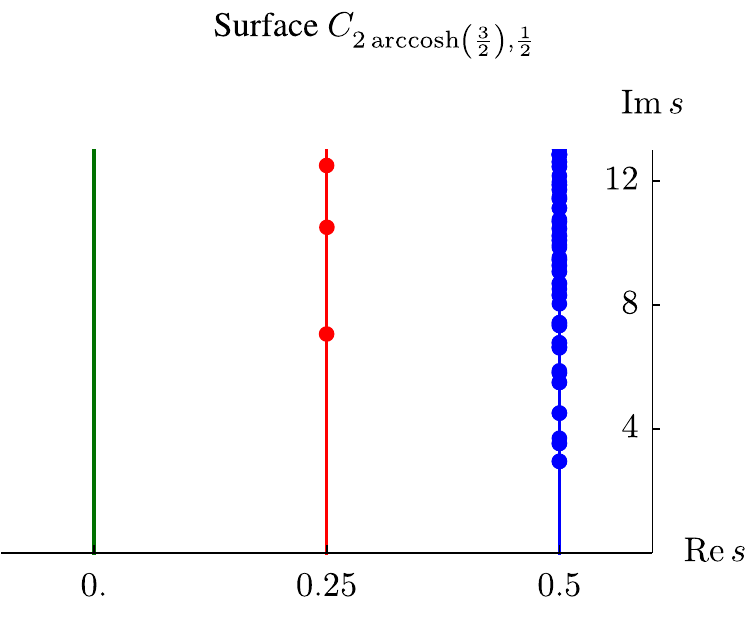}\hfill
\includegraphics{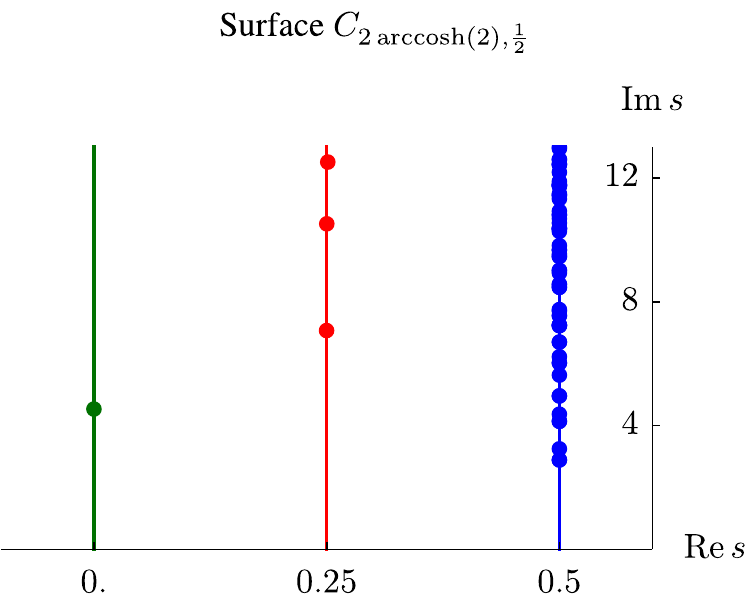}
 \caption{Resonances for $C_{\ell, 1/2}$ for special lengths $\ell$}
 \label{fig15}
 \end{center}
 \end{figure}

\subsubsection*{\underline{$\ell=2\arccosh(3) \approx 3.525494$}}
This surface is isometric to the one with $\ell=2\arccosh\left(\frac{3}{2}\right)$, and our independent numerical results are in full agreement.

\subsubsection{Numerical results,  case 3: $\ell=2\arccosh\left(\frac{3}{2}\right) \approx 1.924847$, varying $\tau$}
The dynamics of resonances is shown in Video \ref{video:5}.

\begin{video}[htb!] 
\centering
\ovalbox{{\href{http://michaellevitin.net/hyperbolic.html#video5}{\texttt{michaellevitin.net/hyperbolic.html\#video5}}}}

\ovalbox{{\href{https://youtu.be/Qk3rmvT7goY}{\texttt{youtu.be/Qk3rmvT7goY}}}}

\caption{The  dynamics of the resonances for $C_{2\arccosh\left(\frac{3}{2}\right) ,1/2}$ as $\tau$ changes in the interval $[0,0.5]$\label{video:5}}
\end{video} 
 
 \subsubsection{Numerical results,  case 4: equal length geodesics,  varying $\tau$ and $\ell=\ell^*(\tau)$}
 
 The dynamics of resonances is shown in Video \ref{video:6}.

\begin{video}[htb!] 
\centering
\ovalbox{{\href{http://michaellevitin.net/hyperbolic.html#video6}{\texttt{michaellevitin.net/hyperbolic.html\#video6}}}}

\ovalbox{{\href{https://youtu.be/7_yBpOxoY9I}{\texttt{youtu.be/7\textunderscore yBpOxoY9I}}}}

\caption{The  dynamics of the resonances for $C_{\ell^*(\tau) ,\tau}$ as $\tau$ changes in the interval $[0,0.489]$\label{video:6}}
\end{video} 

\begin{remark} \label{Remark-comp}
 There are precisely four isomorphism classes of smooth arithmetic surfaces of genus one with one cusp (see \cite{MR698343}, and also \cite{MR702765}). One can use the generators for the four surfaces
 $C_{\arccosh(2),0}$, $C_{\arccosh(3),0}$, $C_{\arccosh(9),0}$, $C_{2 \arccosh\left(\frac{3}{2}\right),\frac{1}{2}}$ and identify them, using \cite[Theorem 4.1]{MR702765}, with the four known arithmetic cases.
 We discovered these special parameters by looking for values of the Fenchel-Nielsen parameters for which the resonances are all along critical lines. The numerical data and the location of the scattering poles then allowed us to conjecture formulae for the scattering matrix. We are very grateful to Andreas Str\"ombergsson, who saw the relation to $\tilde \Gamma_0(N)$ from the formulae and was willing to share his expertise on Atkin-Lehner theory. This made it possible to provide proofs for the corresponding formulae  \eqref{eq:no-twist-C-special1} and \eqref{eq:no-twist-C-special3}.
\end{remark}

\subsection{$D$. The hyperbolic surface of genus zero with three cusps}
 
\subsubsection{Description of the surface}

This surface is unique up to isometry and can be constructed as follows. Take the domain in the upper half space with boundary given by
the the four curves $\gamma_1, \gamma_2, \gamma_3$ and $\gamma_4$ (see  Figure  \ref{fig16}).
Here $\gamma_1$ and  $\gamma_2$ are the two half-circles of radius $\frac{1}{2}$ centered at $z=\frac{1}{4}$ and $z=-\frac{1}{4}$ respectively.
The curves $\gamma_3$ and $\gamma_4$ are the half lines perpendicular to the real axis originating from $z=-\frac{1}{2}$ and $z=\frac{1}{2}$ respectively.
The surface is obtained by identifying $\gamma_1$ and $\gamma_2$, as well as $\gamma_3$ and $\gamma_4$. The three cusps are then located at $z=0$, $z=\frac{1}{2}$,
and at infinity. The surface can also be obtained as a quotient of the upper half space by the subgroup $\Gamma_0(4)$ in $\mathrm{PSL}(2,\R)$ which is generated by the matrices
$\left( \begin{matrix} 1 & 1 \\ 0 & 1 \end{matrix} \right)$ and $\left( \begin{matrix} 1 & 0\\ 4 & 1 \end{matrix} \right)$.
The cusps at $z=0$ and $z=\frac{1}{2}$ can be removed from the surface by cutting along a horocycle (see  Figure \ref{fig16}) and one then obtains two cusps. Each cusp is isometric to a standard cusp of some height. Removing the three cusps in this way one remains with a compact surface with three boundary components. This corresponds to the darker shaded region in Figure  \ref{fig16}. Note that the points $z=\frac{1}{2}$ and $z=-\frac{1}{2}$ (belonging to the compactification of the hyperbolic plane) are identified.

\begin{figure}[htb!]
\begin{center}
\includegraphics{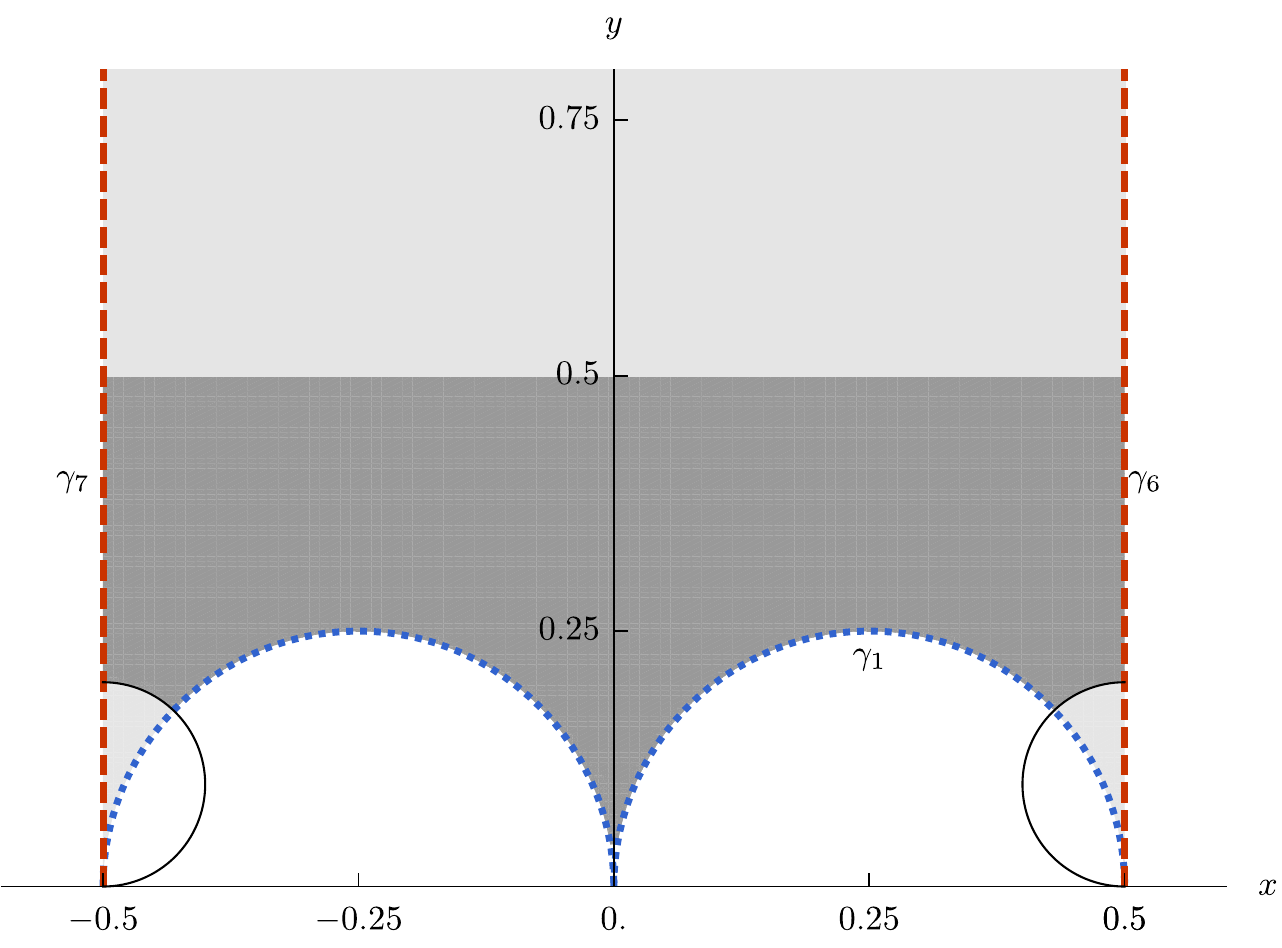}
\caption[Fundamental domain for a hyperbolic surface of genus zero with three cusp]{Fundamental domain for a hyperbolic surface of genus zero with three cusp. The shading indicates the decomposition into cusps and compact part}
 \label{fig16}
 \end{center}
 \end{figure}
  
\subsubsection{Numerical results} Here the scattering matrix is a $3 \times 3$ matrix, and our algorithm computes this reliably. In order to find the resonances we locate the zeros of the determinant
of the scattering matrix and make use of the functional equation. 
Numerically we find that the resonances in this case are of multiplicity three at half the non-trivial roots of the Riemann zeta function, with additional resonances of multiplicity two
at the points
$\frac{\I k \pi}{\log{2}},\; k \in \mathbb{Z} \backslash \{0\}$, see Figure \ref{fig17} and Table \ref{table:genus-0-3cusps}.
Our root finding algorithm finds roots very close to one another in the case of multiplicities. It factors out an already detected root from the function and is therefore able to detect other roots close to the already found one. We can not distinguish numerically between true multiplicities and resonances that are very close to one another.
What we find numerically is in excellent agreement with the known value of the scattering matrix for $\Gamma_0(4)$ \cite{Bruggemann}:
\[
 C(s) =  \frac{1}{2^{2s}-1} \frac{\Lambda(2s -1)}{ \Lambda(2s) }  \begin{pmatrix} 2^{1-2s} & 1-2^{1-2s} & 1-2^{1-2s} \\
 1-2^{1-2s}  & 2^{1-2s} & 1-2^{1-2s} \\ 1-2^{1-2s}  & 1-2^{1-2s} & 2^{1-2s},
   \end{pmatrix}
\]
and the resonances $\frac{\I k \pi}{\log{2}},\; k \in \mathbb{Z} \backslash \{0\}$ are again due to the rational factors in the scattering matrix.

\begin{figure}[htb!]
\begin{center} 
\includegraphics{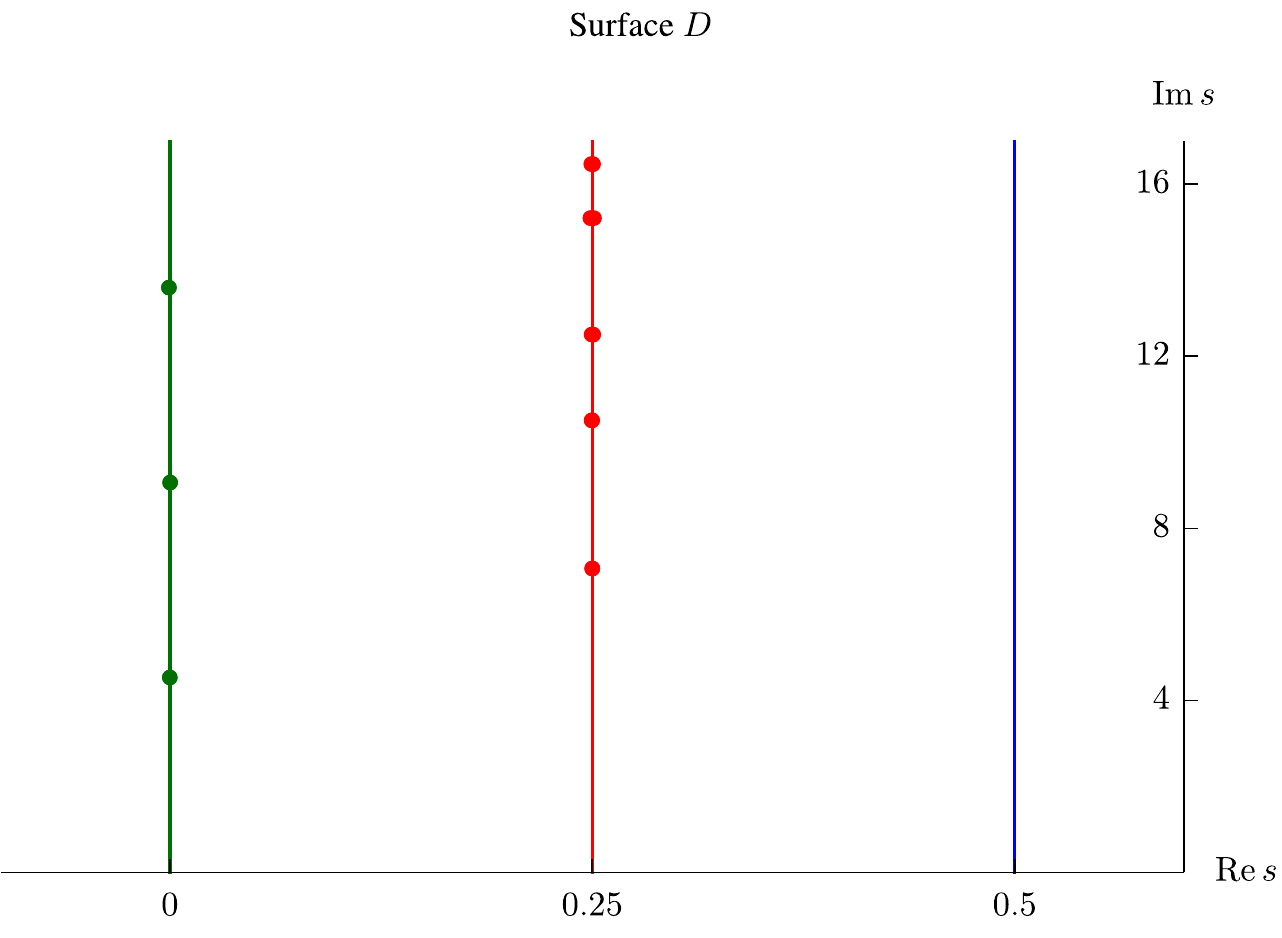}
\caption{Resonances for $D$}
\label{fig17}
\end{center}
\end{figure}  

\subsection*{Acknowledgements}\label{sec:ack}
\addcontentsline{toc}{section}{\nameref{sec:ack}}
 We are very grateful to Andreas Str\"ombergsson for important comments on a draft version of this paper (see Remark \ref{Remark-comp}). We also thank an anonymous referee for many valuable comments and suggestions on an earlier version of the paper.



\clearpage

\begin{thebibliography}{00}
\bibitem{MR2261026}
Avelin, H.
``Deformation of {$\Gamma_0(5)$}-cusp forms.''
 {\em Math. Comp.} {\bf 76} (2007): 361--384.

\bibitem{MR2398794}
Avelin, H.
``Computations of {E}isenstein series on {F}uchsian groups.''
 {\em Math. Comp.} {\bf 77} (2008): 1779--1800.

\bibitem{MR2731549}
Avelin, H.
 ``Numerical computations of {G}reen's function and its {F}ourier coefficients on {$\mathrm{PSL}(2,\mathbb{Z})$}.''
 {\em Exp. Math.} {\bf 19}, issue 3 (2010):335--343.

\bibitem{BSVdata}
Booker,  A. R., Str{\"o}mbergsson, A., and Venkatesh, A.
The data files for \cite{MR2249995} at
  \href{http://www2.math.uu.se/~astrombe/emaass/emaass.html}{\texttt{http://www2.math.uu.se/\textasciitilde{}astrombe/emaass/emaass.html}} (2006).

\bibitem{MR2249995}
Booker,  A. R., Str{\"o}mbergsson, A., and Venkatesh, A.
 ``Effective computation of {M}aass cusp forms.''
 {\em Int. Math. Res. Not.} {\bf 2006} (2006): 71281.

\bibitem{Cuyt2008}
Cuyt, A. A. M., Brevik Petersen, V., Verdonk, B., Waadeland, H., and Jones, W. B.
``\emph{Handbook of Continued Fractions for Special Functions}.''
 Springer (2008).

\bibitem{Bruggemann}
Bruggeman, R., Fraczek, M., and Mayer, D.
 ``Perturbation of Zeros of the Selberg Zeta Function for {$\Gamma_0(4)$}.''
 {\em Exp. Math.} {\bf 22} (2012): 217--242.

\bibitem{CakoniChanillo}
Cakoni, F., and Chanillo, S.
 ``Transmission eigenvalues and the {R}iemann zeta function in
  scattering theory for automorphic forms on {F}uchsian groups of type {I}.''
 {\em Acta Math. Sin,, English Series}  (2019)  (to appear).

\bibitem{MR2164106}
Farmer, D. W., and Lemurell, S.
 ``Deformations of {M}aass forms.''
 {\em Math. Comp.} {\bf 74} (2005): 1967--1982.

\bibitem{greenberg1974discontinuous}
Greenberg, L. (editor)
``Discontinuous Groups and Riemann Surfaces (AM-79): Proceedings of the 1973 Conference at the University of Maryland."
\emph{Annals of Mathematics Studies} series {\bf 79} 
Princeton University Press (1974).

\bibitem{guillope1995upper}
Guillop{\'e}, L.
 ``Upper bounds on the number of resonances for non-compact {R}iemann surfaces.''
 {\em J. Funct. Anal.} {\bf 129}, issue 2 (1995): 364--389.

\bibitem{gutzwiller1983stochastic}
Gutzwiller, M. C.
 ``Stochastic behavior in quantum scattering.''
 {\em Physica D}  {\bf 7}, issues 1-3 (1983): 341--355.

\bibitem{hecht2012new}
Hecht, F.
``New development in {FreeFem++}.''
 {\em J. Num.  Math.} {\bf 20}, no. 3-4 (2012): 251--266.

\bibitem{hecht2012freefem}
Hecht, F.
\href{http://freefem. org}{{``FreeFem documentation''}} (2019).

\bibitem{MR0371818}
Hejhal, D. A.
 ``The {S}elberg trace formula for congruence subgroups.''
 {\em Bull. Amer. Math. Soc.} {\bf 81} (1975): 752--755.

\bibitem{hejhal1976selberg}
Hejhal, D. A.
 {\em The Selberg trace formula for {$\mathrm{PSL}(2, \mathbb{R})$}}.
 Springer (1976).

\bibitem{hejhal92}
Hejhal, D. A.
 ``On eigenvalues of the {L}aplacian for {H}ecke triangle groups.''
 \emph{Adv. Stud. Pure Math. Zeta Functions in Geometry}, N. Kurokawa and T. Sunada, eds. 
Tokyo: Mathematical Society of Japan (1992): 359--408.

\bibitem{hillairet2014hyperbolic}
Hillairet, L. and Judge, C.
``Hyperbolic triangles without embedded eigenvalues.''
 {\em Ann. of Math.} {\bf 187} (2018):301--377.

\bibitem{HowardThesis}
Howard, P.
``Resonance behaviour for classes of billiards on the Poincare half-plane.''
 Technical Report RHUL--MA--2007--6, Royal Holloway University of London (2007).

\bibitem{MR803366}
Huxley, M. N.
``Scattering matrices for congruence subgroups.''
{\em Modular forms ({D}urham, 1983)}, Ellis Horwood Ser. Math.
  Appl.: Statist. Oper. Res. Horwood, Chichester (1984): 141--156.

\bibitem{MR1942691}
Iwaniec, H.
``Spectral methods of automorphic forms''. 
{\em  Graduate Studies in Mathematics} {\bf 53}.
 Amer. Math. Soc., Providence, RI, and Rev. Mat. Iberoam., Madrid, second edition (2002).
  
 \bibitem{JST13} 
Jorgenson, J., Smajlovi\'{c}, L., and Then, H.
``On the distribution of eigenvalues of Maass forms on certain moonshine groups.''   
{\em Math. Comp.} {\bf 83} (2014): 3039--3070, 2014.


\bibitem{kar2013computation}
Kar-Tim, C., Zainuddin, H., and Molladavoudi, S.
 ``Computation of quantum bound states on a singly punctured two-torus.''
 {\em Chinese Phys. Lett.} {\bf 30}, no. 1 (2013): 010304.

\bibitem{Levitin2008}
Levitin, M. and Marletta, M.
 ``A simple method of calculating eigenvalues and resonances in domains with infinite regular ends.''
\emph{Proc. Royal Soc. Edinb.} {\bf 138A} (2008):1043--1065.

\bibitem{MR1029119}
Lee, J. M., and Uhlmann, G.
``Determining anisotropic real-analytic conductivities by boundary measurements.''
 {\em Comm. Pure Appl. Math.} {\bf 42}, no. 8 (1989): 1097--1112.

\bibitem{MR2900555}
Mayer, D., M{\"u}hlenbruch, T., and Str{\"o}mberg, F.
``The transfer operator for the {H}ecke triangle groups.''
 {\em Discrete Contin. Dyn. Syst.} {\bf 32}, no. 7 (2012): 2453--2484.

\bibitem{MR698343}
Maclachlan, C. and Rosenberger, G.
``Two-generator arithmetic {F}uchsian groups.''
 {\em Math. Proc. Cambridge Philos. Soc.} {\bf 93}, no. 3 (1983): 383--391.

\bibitem{MR725778}
M{\"u}ller, W.
``Spectral theory for {R}iemannian manifolds with cusps and a related trace formula.''
 {\em Math. Nachr.} {\bf 111} (1983):197--288.

\bibitem{muller1992spectral}
M{\"u}ller, W.
``Spectral geometry and scattering theory for certain complete surfaces of finite volume.''
 {\em Invent. math.} {\bf 109}, no. 1 (1992):265--305.

\bibitem{MR1226958}
Phillips, R., and Sarnak, P.
 "Automorphic spectrum and {F}ermi's golden rule."
 {\em J. Anal. Math.} {\bf 59} (1992): 179--187.

\bibitem{R69}
Rosenbloom, P. C.
``Perturbation of the zeros of analytic functions. I.''
 {\em J. Approx. Theory} {\bf 2}, issue 2 (1969): 111--126.

\bibitem{selberg1989collected}
Selberg, A.
 {\em Collected papers}, volume~1.
 Springer, 1989.

\bibitem{str}
Str\"{o}mberg, F.
 ``Newforms and spectral multiplicity for $\Gamma_0(9)$.''
 {\em Proc. London Math. Soc. (3)} {\bf 105} (2012):281--310.

\bibitem{str19}
Str\"{o}mberg, F.
``Noncongruence subgroups and Maass waveforms.''
 {\em J. Number Theory} {\bf 199} (2019):436--493.

\bibitem{MR702765}
Takeuchi, K.
 ``Arithmetic {F}uchsian groups with signature {$(1;e)$}.''
 {\em J. Math. Soc. Japan} {\bf 35}, no. 3 (1983): 381--407.

\bibitem{venkov}
Venkov, A.
``A remark on the discrete spectrum of the automorphic {L}aplacian for a generalized cycloidal subgroup of a general {F}uchsian group.''
 {\em J. Math. Sci.} {\bf 52} (1990): 3016.

\bibitem{winkler1988cusp}
 Winkler, A. M.
``Cusp forms and {H}ecke groups.''
 {\em J. Reine Angew. Math} {\bf 386} (1988): 187--204.
\end{thebibliography}

\newpage
\begin{appendices}
\section{Tables of resonances and eigenvalues}

\begin{table}[hbt!]
 \caption[Eigenvalues $\lambda=\frac{1}{4}+t^2$  for the space of odd functions on $A_0$]{Eigenvalues $\lambda=\frac{1}{4}+t^2$  for the space of odd functions on  $A_0$. Data  from \cite{BSVdata} for comparison}
 \label{table:modularodd}
  \begin{center} 
  \begin{tabsize}
  \begin{tabular}{cc}
  \toprule
Computed $t$& Data from \cite{BSVdata}\\
  \midrule
9.53369&9.5336\dots\\ 
12.1730&12.1730\dots\\
14.3585&14.3585\dots\\
16.1381&16.1380\dots\\ 
16.6443&16.6442\dots\\
18.1809&18.1809\dots\\ 
19.4847&19.4847\dots\\
  \bottomrule
  \end{tabular}
  \end{tabsize}
     \end{center}
  \end{table}

\begin{table}[htb!]
\caption{Resonances for $B_r$}
\label{table:artin-even-resonances}
 \begin{center} 
   \begin{tabsize}
\begin{tabular}{ccccc}
  \toprule
    \multicolumn{4}{c}{Computed resonances for}& \multirow[t]{2}{*}{Poles of \eqref{eq:Conj1sqrt3}--\eqref{eq:Conj1}}\\
$B_1=A_0$&$B_{1/\sqrt{2}}$&$B_{1/\sqrt{3}}$&$B_{0.5001}$\\
\midrule
0.2500 + 7.0674 $\I$& 0.2499 + 7.0676 $\I$& 0.2501 + 7.0681 $\I$& 0.2499 + 7.0707 $\I$&$\zeta_1/2\approx$ 0.2500 + 7.0674 $\I$\\
0.2500 + 10.5110 $\I$& 0.2499 + 10.5116 $\I$& 0.2502 + 10.5129 $\I$& 0.2484 + 10.5182 $\I$&$\zeta_2/2\approx$ 0.2500 + 10.5110 $\I$\\
0.2500 + 12.5054 $\I$& 0.2504 + 12.5063 $\I$& 0.2497 + 12.5093 $\I$& 0.2516 + 12.5168 $\I$&$\zeta_3/2\approx$ 0.2500 + 12.5054 $\I$\\
0.2500 + 15.2125 $\I$& 0.2495 + 15.2145 $\I$& 0.2489 + 15.2165 $\I$& 0.2470 + 15.2297 $\I$&$\zeta_4/2\approx$ 0.2500 + 15.2124 $\I$\\
0.2501 + 16.4676 $\I$& 0.2495 + 16.4700 $\I$& 0.2491 + 16.4747 $\I$& 0.2528 + 16.4889 $\I$&$\zeta_5/2\approx$ 0.2500 + 16.4675 $\I$\\
0.2500 + 18.7931 $\I$& 0.2496 + 18.7960 $\I$& 0.2497 + 18.8023 $\I$& 0.2419 + 18.8246 $\I$&$\zeta_6/2\approx$ 0.2500 + 18.7931 $\I$\\
0.2499 + 20.4594 $\I$& 0.2507 + 20.4635 $\I$& 0.2424 + 20.4683 $\I$& 0.2570 + 20.4880 $\I$&$\zeta_7/2\approx$ 0.2500 + 20.4594 $\I$\\
& & -0.0001 + 2.8597 $\I$&& $\pi\I/\log(3)\approx$ 2.8596 $\I$\\
&  -0.0000 + 4.5325 $\I$& & & $\pi\I/\log(2)\approx$ 4.5324 $\I$\\
& & -0.0006 + 8.5796 $\I$& &  $3\pi\I/\log(3)\approx$ 8.5788 $\I$\\
& 0.0002 + 13.5984 $\I$& & &  $3\pi\I/\log(2)\approx$13.5971 $\I$\\
& & 0.0007 + 14.3001 $\I$& & $5\pi\I/\log(3)\approx$14.2980 $\I$\\
  \bottomrule
  \end{tabular}
  \end{tabsize}
  \end{center}
  \end{table}

\begin{table}[htb!]
 \caption[Embedded eigenvalues $\lambda=\frac{1}{4}+t^2$  for the space of even functions on  $B_r$]{Embedded eigenvalues $\lambda=\frac{1}{4}+t^2$  for the space of even functions on  $B_r$. All eigenvalues have multiplicity one. $^*$ denotes eigenvalues for the so called old-forms missed in \cite{winkler1988cusp}}
\label{table:artin-ev}
  \begin{center} 
     \begin{tabsize}
  \begin{tabular}{cccccc}
  \toprule
 \multicolumn{2}{c}{$B_1=A_0$}&  \multicolumn{2}{c}{$B_{1/\sqrt{2}}$}&  \multicolumn{2}{c}{$B_{1/\sqrt{3}}$}\\
 Computed & Data from&  Computed & Data from & Computed & Data from\\
  $t$ &  \cite{BSVdata} &   $t$ &  \cite{hejhal92, winkler1988cusp} &  $t$ &  \cite{hejhal92, winkler1988cusp} \\
  \midrule
13.7798&13.7797\dots & 8.92297&8.92288&5.09885&5.09874\\
17.7387&17.7386\dots &10.9206&10.9204&8.03918&8.03886\\
19.4237&19.4847\dots& 13.7802&13.7798$^*$&9.74450&9.74375\\
&&14.6855&14.6852 &11.3470&11.3464\\
&&16.4044&16.4041&11.8906&11.8900\\
&&17.7394&17.7386$^*$ &13.1362&13.1351\\
&&17.8788&17.8780 &13.7810&13.7798$^*$\\
&&19.1261&19.1254 &14.6278&14.6262\\
&&19.4245&19.4235$^*$ &15.8012&15.7995\\
&&&&16.2727&16.2710 \\
&&&&16.7384&16.7362 \\
&&&&17.5021&17.5006 \\
&&&&17.7413&17.7385$^*$ \\
&&&&18.6501&18.6474 \\
&&&&18.9662&18.9626 \\
&&&&19.4268&19.4235$^*$ \\
&&&&19.8997&19.8961 \\
 \bottomrule
  \end{tabular}
     \end{tabsize}
  \end{center}
  \end{table}

\begin{table}[htb!]
\caption[Resonances for $C_{\ell, 0}$ when $\ell$ is a special length]{Resonances for $C_{\ell, 0}$ when $\ell$ is a special length. The actual computed values for  $C_{\arccosh\left(5\right), 0}$ may differ by one in the last digit from those shown in the first column}
\label{table:genus-one-no-twist-special-ell}
  \begin{center} 
  \begin{tabsize}
  \begin{tabular}{cccc}
  \toprule
  \multicolumn{3}{c}{Computed resonances for}& \multirow[t]{4}{*}{Poles of \eqref{eq:no-twist-C-special1}--\eqref{eq:no-twist-C-special3}}\\
$C_{\arccosh\left(2\right), 0}$&  \multirow[t]{3}{*}{$C_{\arccosh\left(3\right), 0}$}&\multirow[t]{3}{*}{$C_{\arccosh\left(9\right), 0}$}\\
and\\
$C_{\arccosh\left(5\right), 0}$\\ 
  \midrule
 0.2500 + 7.0678 $\I$ & 0.2498 + 7.0680 $\I$ & 0.2500 + 7.0677 $\I$   &$\zeta_1/2\approx$  0.25+ 7.0674 $\I$\\
0.2498 + 10.5130 $\I$& 0.2501 + 10.5127 $\I$& 0.2496 + 10.5117 $\I$&$\zeta_2/2\approx$  0.25 + 10.5110 $\I$\\
0.2507 + 12.5071 $\I$& 0.2495 + 12.5079 $\I$& 0.2495 + 12.5082 $\I$&$\zeta_3/2\approx$  0.25 + 12.5054 $\I$\\
0.2508 + 15.2170 $\I$& 0.2496 + 15.2183 $\I$& 0.2488 + 15.2183 $\I$&$\zeta_4/2\approx$  0.25 + 15.2124 $\I$\\
0.2505 + 16.4737 $\I$& 0.2497 + 16.4742 $\I$& 0.2499 + 16.4745 $\I$&$\zeta_5/2\approx$  0.25 + 16.4675 $\I$\\
\midrule
& &  -0.0000 + 1.9520 $\I$& $\pi\I/\log(5)\approx$ 1.9520 $\I$\\
 0.0000   + 2.8596 $\I$& & & $\pi\I/\log(3)\approx$ 2.8596 $\I$\\
-0.0001 + 4.5324 $\I$& -0.0001 + 4.5326 $\I$& &$\pi\I/\log(2)\approx$ 4.5324 $\I$\\
& & -0.0000 + 5.8562 $\I$&  $3\pi\I/\log(5)\approx$ 5.8559 $\I$\\
-0.0007 + 8.5797 $\I$& & & $3\pi\I/\log(3)\approx$ 8.5788 $\I$\\
& & -0.0005 + 9.7608 $\I$& $5\pi\I/\log(5)\approx$ 9.7599 $\I$\\
0.0041 + 13.6006 $\I$& -0.0034 + 13.6003 $\I$& & $3\pi\I/\log(2)\approx$ 13.5971 $\I$\\
& & -0.0022 + 13.6671 $\I$& $7 \pi\I/\log(5)\approx$ 13.6639 $\I$\\
0.0009 + 14.3046 $\I$& & & $5\pi\I/\log(3)\approx$ 14.2980 $\I$\\
  \bottomrule
  \end{tabular}
  \end{tabsize}
  \end{center}
  \end{table}

\begin{table}[htb!]
 \caption{Resonances  for $C_{\ell, 1/2}$ when $\ell$ is a special length}
 \label{table:genus-one-half-twist-special-ell}
  \begin{center} 
  \begin{tabsize}
  \begin{tabular}{cccc}
  \toprule
  \multicolumn{3}{c}{Computed resonances for}& \multirow{2}{*}{Poles of \eqref{eq:half-twist-C-special1}}\\
  $C_{2\arccosh\left(\frac{3}{2}\right), 1/2}$&$C_{2\arccosh\left(2\right), 1/2}$&$C_{2\arccosh\left(3\right), 1/2}$\\ 
  \midrule
  0.2499+7.0681 $\I$&0.2499+7.0678 $\I$&0.2499+7.0681 $\I$&$\zeta_1/2\approx$ 0.25+7.0674 $\I$\\
  0.2503+10.5123 $\I$&0.2501+10.5128 $\I$&0.2500+10.5121 $\I$&$\zeta_2/2\approx$ 0.25+10.5110 $\I$\\
  0.2499+12.5078 $\I$&0.2512+12.5067 $\I$&0.2492+12.5079 $\I$&$\zeta_3/2\approx$ 0.25+12.5054 $\I$\\
    \midrule
  &0.0+4.5324 $\I$&&$\pi\I/\log(2)\approx$ 4.5324 $\I$\\
  &0.0003+13.6004 $\I$&&$3\pi\I/\log(2)\approx$ 13.5971 $\I$\\
  \bottomrule
  \end{tabular}
  \end{tabsize}
  \end{center}
  \end{table}

 \begin{table}[htb!]
  \caption[Embedded eigenvalues $\lambda=\frac{1}{4}+t^2$ and their multiplicities $\mu(t)$ for $C_{\ell, 0}$ and $C_{\ell, 1/2}$ when $\ell$ is a special length]{Embedded eigenvalues $\lambda=\frac{1}{4}+t^2$ and their multiplicities $\mu(t)$ for $C_{\ell, 0}$ and $C_{\ell, 1/2}$ when $\ell$ is a special length. 
 The actual computed values for  $C_{\arccosh\left(5\right), 0}$ and $C_{2\arccosh\left(3\right), 1/2}$ may differ  in the last digit from those shown in the 
 first and fourth columns, resp. The last two eigenvalues for $C_{\arccosh\left(3\right), 0}$ are shown for comparison with those for $B_{1/\sqrt{2}}$ in Table \ref{table:artin-ev}.
 A subset of these eigenvalues are eigenvalues for the groups $\tilde \Gamma_0(5)$, $\tilde \Gamma_0(6)$, and $\Gamma^3$. These are in good agreement with those computed in \cite{JST13} and \cite{str}}
 \label{table:genus-one-special-ell-evalues}
  \begin{center} 
  \begin{tabsize}
  \begin{tabular}{cccccccccc}
  \toprule
 \multicolumn{2}{c}{$C_{\arccosh\left(2\right), 0}$}&   \multicolumn{2}{c}{$C_{\arccosh\left(3\right), 0}$}& \multicolumn{2}{c}{$C_{\arccosh\left(9\right), 0}$} &\multicolumn{2}{c}{$C_{2\arccosh\left(\frac{3}{2}\right), 1/2}$}& \multicolumn{2}{c}{$C_{2\arccosh\left(2\right), 1/2}$}\\
  \multicolumn{2}{c}{and}&&&&&  \multicolumn{2}{c}{and}\\
  \multicolumn{2}{c}{$C_{\arccosh\left(5\right), 0}$}&&&&&\multicolumn{2}{c}{$C_{2\arccosh\left(3\right), 1/2}$}\\
  $t$&$\mu(t)$&  $t$&$\mu(t)$& $t$&$\mu(t)$&$t$&$\mu(t)$& $t$&$\mu(t)$\\
  \midrule
2.42507 & 1 & 2.89100 & 2 &  2.00968 & 1 &2.95648 & 2 & 2.89101 & 2\\
 3.24141 & 1 & 3.25000 & 1 &  3.44480 & 1 &3.53606 & 2 & 3.25001 & 1\\
 3.97879 & 1 & 4.13811 & 2 & 3.70334 & 1 &3.70339 & 1 & 4.13811 & 2\\
 4.54850 & 1 & 4.36806 & 1 & 4.13245 & 1 &4.51375 & 2 & 4.36809 & 1 \\
 4.64665 & 1 & 4.95729 & 2  & 4.65694 & 1&5.50420 & 2 & 4.95731 & 2  \\
 4.94791 & 1 & 5.62822 & 1  & 4.89729 & 1 &5.81512 & 2 & 5.62824 & 1\\
 5.09888 & 1 & 6.02334 & 2 & 5.05935 & 1 &5.87951 & 1 & 6.02335 & 2\\
 5.19203 & 1 & 6.22332 & 1  & 5.34525 & 1 &6.62069 & 1 & 6.22332 & 1 \\
 5.35557 & 1 & 6.69430 & 2  & 5.87949 & 1 &6.64683 & 2 & 6.69441 & 2 \\
 6.12073 & 1 & 7.22111 & 1 & 6.05422 & 1 & 6.78381 & 2 & 7.22571 & 1\\
 &&\vdots&&&&&&&\\
  &&8.92338&1&&&&&&\\
   &&\vdots&&&&&&&\\
    &&10.9213&1&&&&&&\\
 \bottomrule
    \end{tabular}
    \end{tabsize}
  \end{center}
  \end{table}

\begin{table}[htb!]
 \caption{Resonances  for $D$}
  \label{table:genus-0-3cusps}
  \begin{center} 
  \begin{tabsize}
  \begin{tabular}{cl}
  \toprule
0.2500 + 7.0675  $\I$& \rdelim\}{3}{*}[$\zeta_1/2\approx$  0.25+7.0674 $\I$]\\
0.2500 + 7.0676  $\I$& \\
0.2499 + 7.0680  $\I$& \\
0.2497 + 10.5117  $\I$& \rdelim\}{3}{*}[$\zeta_2/2\approx$  0.25+10.5110 $\I$]\\
0.2496 + 10.5121  $\I$& \\
0.2500 + 10.5125  $\I$& \\
0.2504 + 12.5060  $\I$& \rdelim\}{3}{*}[$\zeta_3/2\approx$  0.25+12.5054 $\I$]\\
0.2506 + 12.5062  $\I$& \\
0.2496 + 12.5078  $\I$& \\
0.2510 + 15.2152  $\I$& \rdelim\}{3}{*}[$\zeta_4/2\approx$  0.25+15.2124 $\I$]\\
0.2489 + 15.2153  $\I$& \\
0.2490 + 15.2154  $\I$& \\
0.2503 + 16.4695  $\I$& \rdelim\}{3}{*}[$\zeta_5/2\approx$  0.25+16.4675 $\I$]\\
0.2504 + 16.4699  $\I$& \\
0.2494 + 16.4724  $\I$& \\
0.2463 + 18.7973  $\I$& \rdelim\}{3}{*}[$\zeta_6/2\approx$  0.25+18.7931 $\I$]\\
0.2454 + 18.7983  $\I$& \\
0.2503 + 18.7984  $\I$& \\
\midrule
-0.0000 + 4.5324  $\I$& \rdelim\}{2}{*}[$\pi \I/\log(2)\approx$ 4.5324 $\I$]\\
-0.0000 + 4.5325  $\I$& \\[4pt]
0.0002 + 9.0654  $\I$& \rdelim\}{2}{*}[$2\pi \I/\log(2)\approx$ 9.0647 $\I$]\\
0.0002 + 9.0657  $\I$& \\[4pt]
-0.0005 + 13.5988  $\I$& \rdelim\}{2}{*}[$3\pi \I/\log(2)\approx$ 13.5971 $\I$]\\
-0.0006 + 13.5997  $\I$& \\[4pt]
-0.0000 + 18.1367  $\I$& \rdelim\}{2}{*}[$4\pi \I/\log(2)\approx$ 18.1294 $\I$]\\
0.0018 + 18.1369  $\I$& \\[3pt]
 \bottomrule
  \end{tabular}
  \end{tabsize}
  \end{center}
  \end{table}
\end{appendices}

\end{document}